\documentclass[12pt]{article}

\usepackage[english]{babel}

\usepackage{amsfonts,epsfig}
\usepackage{amssymb,amsmath,enumerate,float,amsthm}
\usepackage{graphicx,bbm,xcolor}
\usepackage{latexsym}

\usepackage{chngcntr}
\counterwithin{figure}{section}

\newtheorem{definition}{Definition}[section]
\newtheorem{lemma}[definition]{Lemma}
\newtheorem{theorem}[definition]{Theorem}

\newtheorem{remark}[definition]{Remark}
\newcommand{\be}{\begin{equation}}
\newcommand{\ee}{\end{equation}}

\renewcommand{\leq}{\leqslant}
\definecolor{myred}{RGB}{160,0,0}
\definecolor{mygreen}{RGB}{0,160,0}

\newcommand{\ptl}{\partial}
\newcommand{\vph}{\varphi}

\newcommand{\beq}{\begin{equation}}
\newcommand{\eeq}{\end{equation}}

\newcommand{\mathd}{\mathrm{d}}

\newcommand{\tmem}[1]{{\em #1\/}}
\newcommand{\tmop}[1]{\text{  #1 }}
\newcommand{\tmmathbf}[1]{\ensuremath{\boldsymbol{#1}}}

\numberwithin{equation}{section}

\usepackage{hyperref}
\hypersetup{
  colorlinks = true,
   linkcolor = {myred},
   citecolor = {mygreen},
}

\begin{document}

\title{Diffraction by a quarter--plane. Analytical continuation 
of spectral functions}

\author{R.~C.~Assier \& A.~V.~Shanin}
%\address{School of Mathematics, University of Manchester, \\ Oxford Road, Manchester, {\rm %M13 9PL}, UK}
%\extraauthor{A.~V.~SHANIN}
%\extraaddress{Department of Physics (Acoustics Division), Moscow State University, \\ %Leninskie Gory, {\rm 119992}, Moscow, Russia}

%\received{\recd 24 July 2018. \revd ?? Month 2018}

\maketitle

\begin{abstract}
The problem of diffraction by a Dirichlet quarter-plane (a flat cone) in a 3D space
is studied. The Wiener--Hopf equation for this case is derived and involves two unknown 
(spectral) functions depending on two complex variables. The aim of the present work is to build 
an analytical continuation of these functions onto a well-described Riemann manifold and to study their 
behaviour and singularities on this manifold. In order to do so, 
integral formulae for analytical continuation of the spectral functions are derived and used.  
It is shown that the  Wiener--Hopf problem can be reformulated using the concept of additive crossing of branch lines introduced in the paper. Both the integral formulae and the additive crossing reformulation are novel and represent the main results of this work.
\end{abstract}

\section{Introduction and literature review}

%\subsection{Bibliography review for the quarter--plane problem}

\paragraph*{Towards a generalisation of Sommerfeld's solution.}
In the late 19th century, Sommerfeld {\cite{Sommerfeldgerman}} published a
solution to the problem of diffraction by a half-plane by reducing it to a
two-dimensional problem and using a technique now known as Sommerfeld
integrals (see e.g.\ {\cite{babichbook}}). Since, in particular thanks to Jones'
simplification {\cite{Jones1952}}, another popular way of solving this
canonical problem is to make use of the Wiener-Hopf technique
{\cite{Wiener1931,Abrahams,Noble}}, which relies heavily
on the concept of analytical continuation in one complex variable, as well as
on the use of Liouville's theorem. Knowledge of this canonical solution (among
others) inspired the idea of (GTD) Geometrical Theory of Diffraction
{\cite{GTD,Borovikov}}. The aim of the GTD is to describe the
far-field resulting from diffraction by an obstacle exhibiting geometrical
singularities such as edges or corners. In particular, the Sommerfeld solution
allows an analytical description of the far-field diffracted by an edge.

Sommerfeld's solution admits a direct generalisation to the case of diffraction 
by a wedge with scalar equation of motion, and so does the Wiener--Hopf solution 
\cite{Shanin1996,Daniele,Nethercote2018}. 

In the mathematical sense, the Sommerfeld problem is that of diffraction by a half-line 
in a 2D medium. One can have a seemingly natural idea to 
generalize the Sommerfeld or the Wiener--Hopf method to the 3D space 
$(x_1, x_2, x_3)$, in which the scatterer is a quarter-plane (a flat cone)  placed at $x_3 = 0$, $x_1 > 0$, $x_2 >0$. 
The quarter-plane may be hard (resp. soft) in the acoustic formulation, i.e.\ bear Neumann (resp. Dirichlet) boundary conditions 
on both faces. 
Surprisingly, this has proven to be a much more complicated problem, and, though there exists a simple to use closed-form solution of the half-plane problem, this is not yet the case for the quarter-plane.
Such a problem can be treated as a {\em canonical diffraction problem}, 
and its solution would extend the applicability of the composite 
diffraction methods based on the locality principles, such as GTD or PTD (physical theory of diffraction). 
This problem is the subject of the present paper. 

\paragraph*{The quarter-plane as a flat cone.}
The most classic approach (see {\cite{satterwhite}}) 
uses the observation that a quarter-plane is a degenerate elliptic cone. 
Thus, it is possible to use
separation of variables and  sphero-conal coordinates in order to
express the solution as a multipole series involving Lam{\'e} functions (see
{\cite{Erdelyi1955}} for an in-depth description of these functions). However,
the convergence of such series is very poor at high frequencies which made
this approach difficult to use. A curious confirmation of the low usability 
of the series solution is the fact that URSI Commission~B awarded a prize 
in~2005 for an accurate approximate solution of this problem.  
The winning approach, {\cite{Klinkenbusch2005}}, was focused on enhancing convergence of the expansion series (see {\cite{blume}} and {\cite{Blume1999}} for details). 
We should note here also an unusual work, {\cite{Budaev2005a}}, probably inspired by the 
URSI prize appeal, where the author make use of the Feynman-Kac theorem in order to express the corner-diffracted part of the field as a mean calculated from realisations of three simple stochastic differential equations.   

An elegant approach applicable to diffraction by an arbitrary cone and particularly 
to a flat cone has been introduced in {\cite{smy1}} and {\cite{Smyshlyaev1993}}.
One converts the series into a contour integral by means of Bessel--Watson transform, 
and then deforms the contour to obtain fast convergence of the integral.
The integrand contains a Green's function for the        
Laplace-Beltrami operator on a sphere with a cut, which can be 
obtained, generally, by solving an integral equation.
Within this approach, the concept of the {\em oasis\/} domain has been 
introduced, i.e.\ of the domain in the 3D space where there are no scattered waves except 
the spherical wave diffracted by the tip of the cone. The position of the observation point in the oasis guarantees fast convergence of the resulting integral for the 
diffraction coefficient. Outside the 
oasis, the field contains also some waves diffracted by the edges of the
quarter-plane and some reflected waves, and the integral for the diffraction coefficient diverges. One can, however, regularize it by using Abel-Poisson type methods ({\cite{babich95}}, {\cite{Babich1996}}). 

Finding of the diffraction coefficient, i.e.\ the angular dependence of the 
diffracted spherical wave, is a complicated task, while other wave components 
can be found by relatively simple methods. 
An accurate description of such components can be given
using GTD for example {\cite{Assier2012b}}, Sommerfeld integrals
{\cite{Lyalinov2013,Lyalinov2015}}, or using ray asymptotics on a sphere with a cut
{\cite{Shanin2012}}. 
An interesting link between the spherical diffracted wave and the other diffracted
wave components can be found in {\cite{BorovikovPolyhedra}}. It says that the diffraction coefficient of the spherical wave emanating from a corner can be written as the solution of a PDE in a unit ball, for which the boundary data on the surface is provided by the other wave components, which can be found explicitly.

The idea of a sphere with a cut arises from separation of
variables in spherical coordinates. In particular, such approach (see for
example {\cite{jansen}}) allows one to describe the asymptotic behaviour of
the solution in the near-field, close to the corner. The field behaves like
$r^{\nu_1 - 1 / 2}$, where $r$ is the distance to the corner and $\nu_1 =
\sqrt{\lambda_1 + 1 / 4}$, where $\lambda_1$ is the first eigenvalue of the
Laplace-Beltrami operator on the sphere with a cut. The spectral analysis of
such operator is in itself a very interesting topic (see for example
{\cite{SLEEMAN,Keller1999}} as well as {\cite{Assier2016}} and
references therein).

Building on the ideas of {\cite{smy1}} and the concept of embedding formulae (introduced in
{\cite{williams}} and developed further in {\cite{embedding}}) and edge
Green's functions, more formulae, coined {\tmem{modified Smyshlyaev formulae}},
were introduced (see {\cite{shanin1,shanin2,Assier2012,assierphd,Valyaev2011}}) 
and showed to be naturally
convergent in a region much larger than the oasis zone, which makes it one of the most successful current techniques. In fact this zone of
convergence is the zone where no secondary diffracted wave is present in the
far-field. This allowed for a fast computation of the diffraction coefficient for
a wide range of incidence and observation directions in the case of Dirichlet
and Neumann boundary conditions. Still, computation of the diffraction coefficient in the areas where doubly diffracted waves exist remains problematic and no methods have so far succeeded in doing so.

\paragraph*{Generalisations of the Wiener--Hopf method.}
The idea of generalisation of the usual (1D) Wiener--Hopf method seems 
attractive but is far from trivial and, as we believe, a proper generalisation
demands development of substantially new methods. A ``safe'' approach, however, 
is described in {\cite{Meister1988}}, where the 1D Wiener--Hopf method is applied
successively to different coordinates, and the resulting solution becomes expressed 
as a result of infinite number of such operations. 

The difficulty with the 2D Wiener--Hopf method is as follows. In the 1D case, 
the Fourier transform of a function with a support on the positive half-axis $x$ is analytic in the upper half-plane of the Fourier 
variable $\xi$, and vice versa (see e.g.~{\cite{Noble}}).
In the 2D case, a direct generalisation of this fact is valid, namely, a function 
having a support in the quadrant $x_1 > 0$, $x_2 >0$, after a Fourier transform, 
gives a function analytic in the domain ${\rm Im}(\xi_1) > 0$, ${\rm Im}(\xi_2) > 0$   
(see (\ref{eq:defdoubleFourier}) for the definition of the Fourier transform and Theorem \ref{th:FTth2.1} for a rigorous statement of this result).
The inverse statement is also true. The problem is that there exists no such 
a simple criterion for a function having its support in the complement of 
the quadrant, i.e.\ in $x_1 < 0$ or $x_2 < 0$. This problem does not arise in the 1D 
case since a complement of a half-line is also a half-line. 

Some generalisations of the Wiener--Hopf method can be built if the kernel has special properties. Namely, if the kernel can be factorised into two kernels with supports\footnote{Here by supports, we mean the supports of their inverse Fourier transform} 
in two quadrants, then the generalisation is straightforward {\cite{kakichev1968,Rabinovich1967}}. In {\cite{Rabinovich1967}} the restriction mentioned above has been slightly weakened to allow for kernels that can be written as a sum of two or three functions, the supports of which being restricted to a quadrant. 
Some special kernels (products of functions depending on a single variable)
have been studied in {\cite{Moraru1969}}.
Unfortunately, the kernel (\ref{eq:kerneldefqp}) to which the quarter-plane diffraction problem 
is reduced does not fit these restrictions. An attempt to factorise such a kernel was made in {\cite{Vasiliev2000}} and called \textit{wave factorisation}, but we think that it is erroneous. The reason for this will be made much clearer in future work, when we will introduce the concept of \textit{bridges and arrows} for surface of integration visualisation in $\mathbb{C}^2$.

%We should note here the work {\cite{Malyshev1971}}, in which, for kernels of a very 
%restricted class, a mathematically elegant method has been applied, reducing the 
%problem to studying of the set of zeros of the symbol.  

It has to be mentioned that other attempts have been made and were somehow
unsuccessful. One of the most famous such attempt is Radlow's work
({\cite{radlow}} and {\cite{Radlow1965}}) on an extension of the Wiener-Hopf
technique to two complex variables, but his formal closed-form solution led to
the wrong type of near-field behaviour and was hence considered erroneous by
the community (see for example {\cite{meister}}). A formal description of
where his proof went wrong (the solution does not satisfy the boundary condition) was given in {\cite{albani}}, and, more recently,
other technical reasons have been given in {\cite{Assier2018}},
showing along the way however that in the far-field and in the case of
Dirichlet conditions, Radlow's solution led to surprisingly accurate
results.

To conclude this introduction, we can say that getting a closed-form solution 
of the quarter-plane diffraction problem is still a challenging theoretical and practical 
task. Particularly, a generalisation of the Wiener--Hopf method to the case of two complex 
variables would be a considerable achievement.

%\paragraph{Pure mathematics approach.}

%Though research in {\tmem{applied}} mathematics has been vivid on the topic
%for many years, it is important to recognise that many {\tmem{purer}}
%mathematicians have also been interested. In particular by developing the
%field of Toeplitz operators ({\cite{Strang1970}}\tmcolor{red}{, Malysev,
%Radlow 77?}) and the related Wiener-Hopf operators {\cite{Meister1988}}.
%Though as far as we are aware, this did not lead to any easy way of evaluating
%the diffracted field or the diffraction coefficient.

%{\tmstrong{AS:I somehow don't think that Moraru and Rabinovitch fit in that
%paragraph, but I let you decide, since the papers are in Russian. I also need
%to find a place for Cheeger \& Taylor {\cite{CheegerTaylor}},
%{\cite{Cheeger1982}}, and also the weird Budaev \& Bogy {\cite{Budaev2005a}}
%paper using random walk.}}

%\paragraph{Just notes before finalizing}

\paragraph*{Aim of the paper.} The main aim of the present work, is to establish the theoretical framework needed to make progress with a two complex
variables approach of the Wiener--Hopf technique. After deriving the functional
equation associated to the quarter-plane problem, containing two unknown
functions $\tilde{W}$ and $\tilde{U}'$ of two complex variables, we will
endeavour to explicitly exhibit domains where these functions can be
analytically continued. This is motivated by the fact that, as illustrated in section \ref{sec:motivation}, a remarkable characteristic of a usual 1D Wiener--Hopf equation is that such domains can be found \textit{without} solving the problem explicitly.  The ideas and methods introduced in this work can be applied to plane sectors with any interior angle (not just $\pi/2$), but the presence of many multiple edge diffractions would make the exposition quite cumbersome. The main results of this work are:
\begin{itemize}

\item the analytical continuation integral formulae given in Theorems \ref{th:th3} and \ref{th:th34}, allowing to find the domain of analyticity of the spectral functions as given in Theorems \ref{th:th1step1}, \ref{th:th2step1} and \ref{th:th7};
\item the reformulation of the spectral problem given in Section \ref{sec:reformulationadditive} using the novel concept of additive crossing of branches defined in Definition \ref{def:additivecrossing}.
\end{itemize}

In the future, though not in the present paper, we will show how the additive crossing property will lead naturally to the appearance of a $r^{\nu_1-1/2}$ behaviour at the corner. Moreover, the knowledge of the expected singularities on the whole of the Riemann manifold will help us to construct an approximation solution, not dissimilar to Radlow's ansatz, but hopefully more accurate.

\paragraph*{Plan of the paper.}
%We will start by giving a very brief overview of some remarkable facts of the
%complex analysis in two complex variables in Section~\ref{sec:noteson2Dcomplex}. 
In
Section~\ref{sec:functionaleq}, we will derive the functional equation using
Green's theorem, which is a different approach to that used in
{\cite{Assier2018}} for example, and we will introduce the concept of
1/4-based and 3/4-based functions, allowing for a concise spectral formulation
of the physical problem. In Section~\ref{sec:analyticalcontinuation}, we will show that the
unknown functions of the functional equation can be analytically continued in
some larger domains, the union of which constitute the whole of $\mathbbm{C}^2$
except some cuts. This task will be based on the use of explicit analytical
continuation integral formulae. In Section~\ref{sec:additivecrossing}, we will show that the
unknown function $\tilde{U}'$ has a very particular behaviour about two of
its branch sets. We call this behaviour {\tmem{additive crossing}} and this
allows us to rewrite the spectral formulation in a very different way to
that of Section~\ref{sec:functionaleq}, which we will use in further work in order
to derive specific results.

%%%%%%%%%%%%%%%%%%%%%%%%%%%%%%%%%%%%%%%%%%%%%%%%%%%%%%%%%%%%%%%%%%%%

\section{The quarter-plane functional problem} \label{sec:functionaleq}
%Functional problem for diffraction by a quarter-plane

\subsection{Formulation of the physical problem}\label{sec:formulationphysical}

Consider a 3D space $(x_1 , x_2 , x_3 )$ The scatterer is a quarter-plane 
(QP hereafter) $x_3 = 0$, $x_1 > 0$, $x_2 > 0$, as shown in Fig.~\ref{fig:geompluscarcasse} (left). 
Everywhere outside the scatterer the Helmholtz equation is valid for the total field: 
\begin{equation}
\Delta u^{\rm t} + k^2 u^{\rm t} = 0 ,
\label{eq:Helmholtz}
\end{equation}
where $k$ is a wavenumber parameter, such that $k^2 = 1 + i \epsilon$, 
$\epsilon$ being a small positive parameter. The total field is the sum of the incident and scattered fields: 
\[
u^{\rm t} = u^{\rm in} + u.
\]   
The incident field is a plane wave given by 
\[
u^{\rm in} = \exp \left\{ i \left( k_1 x_1 + k_2 x_2 - 
\sqrt{k^2 - k_1^2 - k_2^2} \, x_3 \right)  \right\}. 
\]    
The boundary condition is of inhomogeneous Dirichlet type on both faces of the QP:
\begin{equation}
u (x_1, x_2, 0) = - \exp \{ i (k_1 x_1 + k_2 x_2 ) \} 
\qquad \text{ if }  x_1 > 0  \text{ and } \ x_2 > 0.  
\label{eq:BCDi}
\end{equation}
We assume that ${\rm Re}(k_1) > 0$, ${\rm Re}(k_2) > 0$, and both values 
have a small positive imaginary part. 
The condition ${\rm Re}(k_1) > 0$, ${\rm Re}(k_2) > 0$ is rather restrictive, since it 
means that the incident wave cannot produce a doubly diffracted wave. However in this paper, we only have the ambition to introduce the theoretical framework that will allow us to make progress. The extension of this work to all possible incidence directions will be the topic of future work.

The scattered field $u$ should
also obey some radiation conditions at infinity in the limiting absorption
form, that is, in the case of $\epsilon > 0$, we need to have $u$ tends to $0$
as $r = (x_1^2 + x_2^2 + x_3^2)^{1 / 2}$ tends to $\infty$.

%{\bf RA, when we will be strong enough we will handle the case 
%${\rm Re}(k_{1,2}) < 0$, I hope}. Ok I've added a sentence about this.

Considering the symmetry of the problem allows us to reduce in a standard way the problem for the scattered field to a boundary value problem in the half space with mixed boundary condition. Here in particular, we find that the scattered field is symmetric
and hence we can reduce the problem to $x_{1, 2} \in \mathbb{R}$ 
and $x_3 \in
\mathbb{R}^+$, subject to the Neumann boundary condition
\begin{equation}
   \frac{\partial u}{\partial x_3} (x_1, x_2, 0) = 0 \quad \text{ if } \quad
    x_1 < 0 \text{ or }  x_2 < 0  .
  \label{eq:BCN}
\end{equation}
Hence, in terms of regularity, we just need $u$ to be infinitely smooth for
$x_3 > 0$ and continuous as $x_3 \rightarrow 0^+$. For the problem to be well
posed, we also need to impose the so called Meixner conditions, or edge
conditions stipulating that close to the edges, the field should decay like
the square root of the distance to the edge. Hence there must exist two ``edge
functions'' $e_1$ and $e_2$ such that
\begin{equation}
    \left|u^{\rm t} (x_1, x_2, x_3)\right| < e_1 (x_1) (x_2^2 + x_3^2)^{1 / 4} 
    \text{ as }  (x_2^2
    + x_3^2)^{1 / 2} \rightarrow 0  \text{ for a fixed } x_1 > 0, 
\label{eqI0102}
\end{equation}    
\begin{equation}
    \left|u^{\rm t} (x_1, x_2, x_3)\right| < e_2 (x_2) (x_1^2 + x_3^2)^{1 / 4} 
    \text{ as }  (x_1^2
    + x_3^2)^{1 / 2} \rightarrow 0  \text{ for a fixed } x_2 > 0, 
\label{eqI0103}
\end{equation}        
%where $u^{\rm s}$ is the symmetrical part of the total field: 
%\[
%u^{\rm s} (x_1, x_2 , x_3) = u^{\rm t}(x_1, x_2 , x_3) 
%- \frac{1}{2} u^{\rm in} (x_1, x_2 , x_3)  
%+ \frac{1}{2} u^{\rm in} (x_1, x_2 , -x_3).
%\]

There is another Meixner condition that should be imposed at the vertex.
There should exist a ``corner function'' $c$ such that
\begin{equation}
    \left|u^{\rm t} (x_1, x_2, x_3)\right| < c (\theta, \varphi) r^{\lambda}  
    \text{ as }  r \rightarrow 0 
    \text{ for some } \lambda > -1/2 ,
  \label{eq:vertex}
\end{equation}
where $(r, \theta, \varphi)$ are the usual spherical coordinates, $\theta$ and
$\varphi$ being the polar and azimuthal angles respectively. A function $u$
that satisfies (\ref{eq:Helmholtz})--(\ref{eq:vertex}) and the radiation
condition will be called a solution to the quarter-plane problem with a 
plane wave incidence.

%%%%%%%%%%%%%%%%%%%%%%%%%%%%%%%%%%%%%%%%%%%%%%%%%%%%%%%%%%%%%%%%%%%%%%%%%%%%%%%%

\subsection{Derivation of the functional equation}

For any function $\phi (x_1, x_2, x_3)$, define its double Fourier
transform $\mathfrak{F} [\phi]$ by
\begin{eqnarray}
  \mathfrak{F} [\phi] (\xi_1, \xi_2, x_3) & = & \int_{- \infty}^{\infty}
  \int_{- \infty}^{\infty} \phi (x_1, x_2 , x_3) e^{i (\xi_1 x_1 + \xi_2 x_2)}
  \, \mathd x_1 \mathd x_2 . 
  \label{eq:defdoubleFourier} 
\end{eqnarray}
Subsequently, for any function $\tilde{\Phi}(\xi_1,\xi_2)$, we define its inverse Fourier transform $\mathfrak{F}^{-1} [\tilde{\Phi}]$ by 
\begin{eqnarray}
  \mathfrak{F}^{-1} \left[\tilde{\Phi}\right] (x_1, x_2)  &=& \frac{1}{4\pi^2} 
  \int_{- \infty}^{\infty}
  \int_{- \infty}^{\infty} 
  \tilde{\Phi} (\xi_1, \xi_2 ) e^{- i (\xi_1 x_1 + \xi_2 x_2)}
  \, \mathd \xi_1 \mathd \xi_2 .
  \label{eqI0105} 
\end{eqnarray}
Define also the \emph{kernel} function $\tilde{K}$  by
\begin{eqnarray}
  \tilde{K} (\xi_1, \xi_2)  \equiv (k^2 - \xi_1^2 - \xi_2^2)^{- 1 / 2} . \label{eq:kerneldefqp}
\end{eqnarray}
For real $\xi_1$, $\xi_2$ such that $\xi_1^2 + \xi_2^2 < 1$ the square root 
value is taken to be close to real positive.

Consider the domain $\Omega_{R, \varepsilon}$ consisting of the upper
hemisphere of radius $R > 0$ centred at the vertex, from which a sphere of
radius $\varepsilon > 0$ (around the vertex) has been removed. Hence the
boundary of $\Omega_{R, \varepsilon}$ consists of the upper surface of the big
sphere $S_R^+$ and a ``bottom lid'' made of two planar, and one spherical
surfaces as shown in Fig.~\ref{fig:geompluscarcasse} (centre).

\begin{figure}[h]
\centering
  \includegraphics[width=0.3\textwidth]{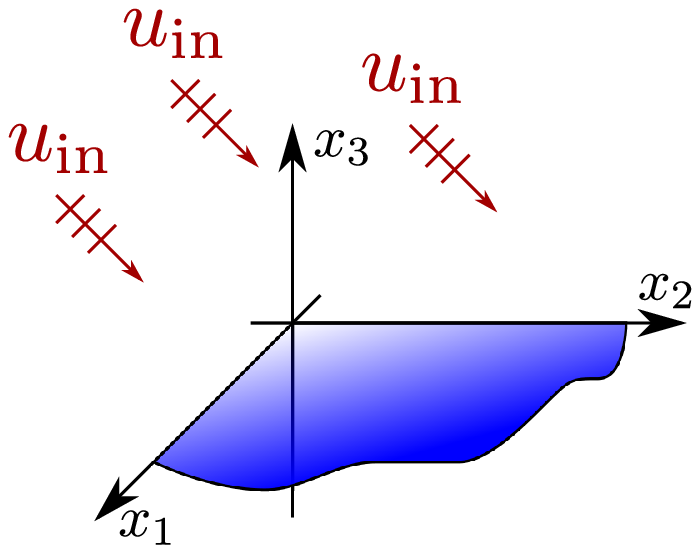}\quad\includegraphics[width=0.3\textwidth]{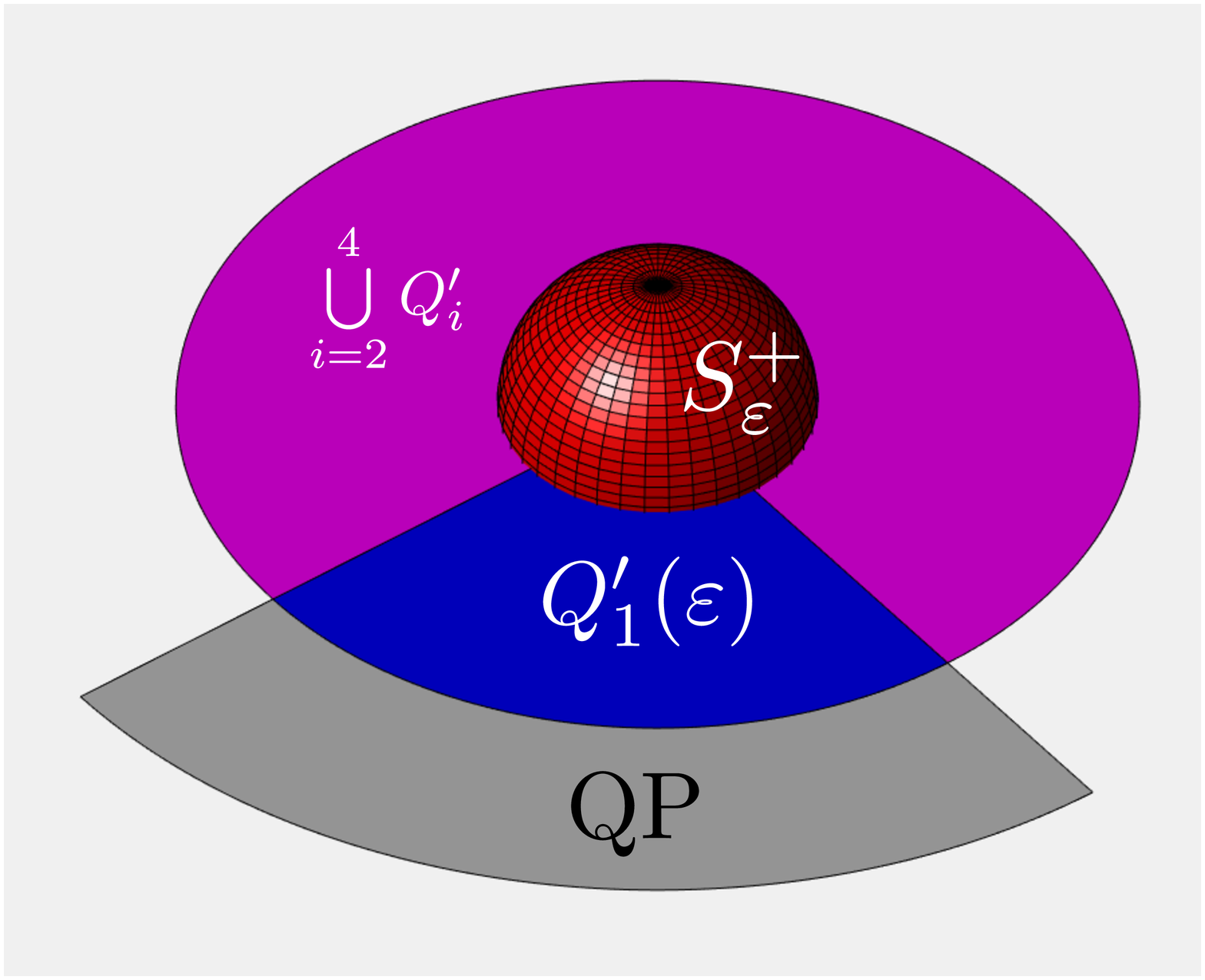}\quad\includegraphics[width=0.3\textwidth]{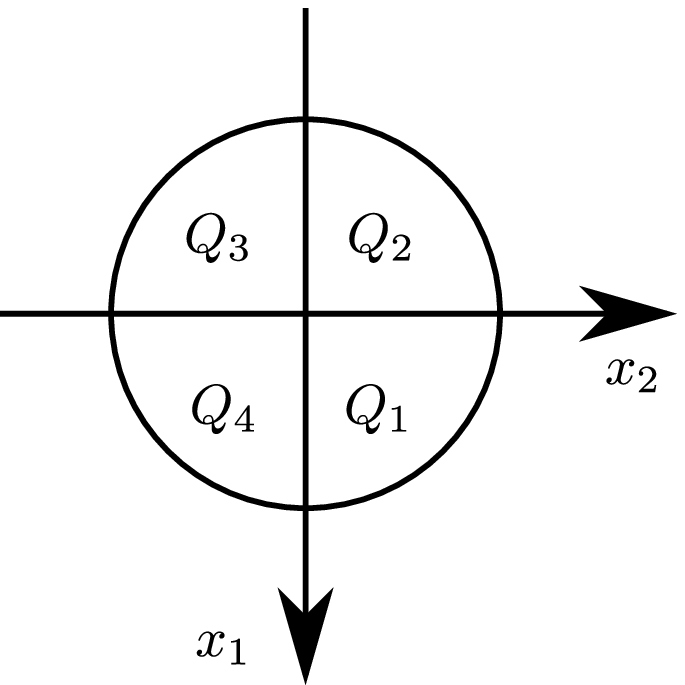}
  \caption{(left) Geometry of the problem. (centre) ``Bottom lid'' of
  $\Omega_{R, \varepsilon}$. The quarter-plane is in grey. (right)
  Illustration of the $Q_i$ quadrants.}\label{fig:geompluscarcasse}
\end{figure}

For any test function $w$ that satisfies the Helmholtz equation, we can take
the volume integral over $\Omega_{R, \varepsilon}$ of the quantity $w \Delta u
- u \Delta w$. On one hand it is zero because of the Helmholtz equation, and
on the other hand this can be expressed as a surface integral over $\partial
\Omega_{R, \varepsilon}$ by Green's theorem. Hence we get
\begin{eqnarray}
  \iint_{\partial \Omega_{R, \varepsilon}} \left( w \frac{\partial
  u}{\partial n} - u \frac{\partial w}{\partial n} \right) \, \mathd S & = & 0,
  \label{eq:initialGreen}
\end{eqnarray}
for any $\varepsilon > 0$, where $\partial / \partial n$ is the normal
derivative, where the unit normal ${\bf n}$ is chosen to be incoming
towards the volume $\Omega_{R, \varepsilon}$, and the test function $w$ is
chosen to be
\begin{eqnarray}
  w & = & e^{i (\xi_1 x_1 + \xi_2 x_2 + x_3 / \tilde{K} (\xi_1, \xi_2))}, \label{eq:defofw}
\end{eqnarray}
for some $\xi_1$ and $\xi_2$. First, we take real $\xi_1$ and $\xi_2$ 
and then continue the results analytically. 
Note that for all real $\xi_1$, $\xi_2$ the value $\tilde K^{-1} (\xi_1, \xi_2)$
has a positive imaginary part. It is small (due to $\epsilon$) 
for $\xi_1^2 + \xi_2^2 < 1$, and not small for $\xi_1^2 + \xi_2^2 > 1$. 
Thus, $w$ decays exponentially as $x_3 \to \infty$. 
 
Note that there is no need to exclude the edges
of the quarter-plane from our volume of integration since thanks to Meixner
conditions the integrals are convergent there. 
Hence, thanks to the exponential decay of $w$, for any $\epsilon > 0$, we have
\begin{eqnarray}
  \lim_{R \rightarrow \infty} \iint_{S_R^+} \left( w \frac{\partial
  u}{\partial n} - u \frac{\partial w}{\partial n} \right) \, \mathd S & = &
  0.  \label{eq:limitSRzero}
\end{eqnarray}

Hence, if we call $Q_i' (\varepsilon)$ the part of the quadrant $Q_i$ (see
Figure \ref{fig:geompluscarcasse}) from which the disk of radius $\varepsilon$
is excluded, and $S_{\varepsilon}^+$ the upper small spherical surface around
the vertex, taking the limit as $R \rightarrow \infty$ in
(\ref{eq:initialGreen}), using the fact %that $u = 0$ on $Q_1'$ and $\partial u
%/ \partial n = 0$ on $Q_i'$ for $i \neq 1$ and the fact
that
${\bf n}={\bf e}_r$ on $S_{\varepsilon}^+$ and
${\bf n}={\bf e}_{x_3}$ on $Q_i'$ leads to
\begin{eqnarray}
  \iint_{S_{\varepsilon}^+} \left( w \frac{\partial u}{\partial r} - u
  \frac{\partial w}{\partial r} \right) \, \mathd S+\iint_{\bigcup\limits_{i=1}^{4} Q_i'} \left(w
  \frac{\partial u}{\partial x_3}-u \frac{\partial w}{\partial x_3}\right)\, \mathd S & = & 0,
  \label{eq:secondGreenepsilon}
\end{eqnarray}
for any $\varepsilon > 0$.

Due to the vertex Meixner condition, one can take the limit $\varepsilon \rightarrow 0$
in (\ref{eq:secondGreenepsilon}). As a result, the integral along $S_{\varepsilon}^+$ goes to zero.
The integrals in the second term in the left
of (\ref{eq:secondGreenepsilon}) become Fourier transforms of $u$ and $\ptl u/ \ptl x_3$. Define 
\begin{eqnarray}
  \tilde{U} (\xi_1, \xi_2) =\mathfrak{F} [u] (\xi_1, \xi_2, 0^+) & \tmop{and}
  & \tilde{W} (\xi_1, \xi_2) =\mathfrak{F} \left[ \frac{\partial u}{\partial
  x_3} \right] (\xi_1, \xi_2, 0^+) .  
  \label{eq:defUtilde}
\end{eqnarray}
It follows from (\ref{eq:secondGreenepsilon})
that 
\begin{equation}
  \tilde{K} (\xi_1, \xi_2) \tilde{W}(\xi_1,\xi_2) = 
  i \tilde{U} (\xi_1, \xi_2) , 
  \label{eq:functionalequation}
\end{equation}
which is the main functional equation for the QP diffraction problem. 
Indeed, this agrees with the functional equation from \cite{radlow} and from other 
papers dedicated to the problem. 

As it is, equation (\ref{eq:functionalequation}) is defined for real $\xi_1 , \xi_2$. 
The real plane $(\xi_1 , \xi_2)$, on which $\tilde U$ and $\tilde W$
are defined by (\ref{eq:defUtilde}) is said to belong to the the ``physical sheet''. Later on we 
will study analytical continuation of $\tilde U$ and $\tilde W$ to other ``sheets''.

%%%%%%%%%%%%%%%%%%%%%%%%%%%%%%%%%%%%%%%%%%%%%%%%%%%%%%%%%%%%%%%%%%%%%%%%%%%%%%%%

\subsection{1/4-based and 3/4-based functions}

Equation (\ref{eq:functionalequation}) contains two unknown functions. Our aim is to convert this equation into a formulation close 
to the traditional Wiener--Hopf equation. For this, we need to take into account some additional properties 
of the unknown functions $\tilde U$ and $\tilde W$. We call these properties 
1/4-basedness and 3/4-basedness. 

Define the quadrants $Q_1, Q_2, Q_3, Q_4$ 
of the plane $(x_1, x_2)$
as it is shown in Fig.~\ref{fig:geompluscarcasse} (right).  Now, upon writing $\boldsymbol{x}=(x_1,x_2)$ and $\boldsymbol{\xi}=(\xi_1,\xi_2)$, note that for any function $\phi(\boldsymbol{x})$, we can write $\mathfrak{F}[\phi]=\mathfrak{F}_{1/4}[\phi]+\mathfrak{F}_{3/4}[\phi]$, where the Fourier operators $\mathfrak{F}_{1/4}$ and $\mathfrak{F}_{3/4}$ are defined by
\begin{eqnarray*}
\mathfrak{F}_{1/4}[\phi](\boldsymbol{\xi})=\iint_{Q_1} \phi(\boldsymbol{x}) e^{i\boldsymbol{x}\cdot\boldsymbol{\xi}}\,\mathd \boldsymbol{x} & \text{and} & \mathfrak{F}_{3/4}[\phi](\boldsymbol{\xi})=\iint_{\bigcup\limits_{i=2}^{4} Q_i} \phi(\boldsymbol{x}) e^{i\boldsymbol{x}\cdot\boldsymbol{\xi}}\,\mathd \boldsymbol{x},
\end{eqnarray*}
and called $1/4$-range and $3/4$-range Fourier transform respectively.
According to (\ref{eq:BCN}), the function $\tilde W$ is a Fourier transform of a function that
is not equal to zero only on the quadrant $Q_1$. Similarly, according to (\ref{eq:BCDi}), the function $\tilde U$ can be represented as follows: 
\begin{equation}
\tilde U(\xi_1 , \xi_2) = \frac{1}{(\xi_1 + k_1) (\xi_2 + k_2)} + \tilde U'(\xi_1 , \xi_2).
\label{eqI0104}
\end{equation} 
The first term is the $\mathfrak{F}_{1/4}$ transform of (\ref{eq:BCDi}), while 
$\tilde U'$
is the $\mathfrak{F}_{3/4}$ transform of $u(x_1 , x_2 , 0)$.

%{\bf Definition 1. }
\begin{definition}
A function $\tilde{\Phi}(\xi_1 , \xi_2)$ is called \textbf{1/4-based\/} if its inverse Fourier transform $\mathfrak{F}^{-1} [\tilde{\Phi}] (x_1 , x_2)$
is equal to zero if $x_1 < 0$ or $x_2 < 0$ (i.e.\ on quadrants $Q_2, Q_3, Q_4$). In other words, $\tilde{\Phi}$ is 1/4-based if there exists a function $\phi$ such that $\tilde{\Phi}=\mathfrak{F}_{1/4}[\phi]$. 
\end{definition}

\begin{definition} 
A function $\tilde{\Phi}(\xi_1 , \xi_2)$ is called \textbf{3/4-based\/} if its inverse Fourier transform $\mathfrak{F}^{-1} [\tilde{\Phi}] (x_1 , x_2)$
is equal to zero if $x_1 > 0$ and $x_2 > 0$ (i.e.\ on quadrant $Q_1$). In other words, $\tilde{\Phi}$ is 3/4-based if there exists a function $\phi$ such that $\tilde{\Phi}=\mathfrak{F}_{3/4}[\phi]$. 
\end{definition}
Thus, it is clear from the discussion above, that $\tilde W$ is 1/4-based and $\tilde U'$ is 3/4-based. As discussed briefly in the introduction, the concept of a 1/4-based function causes no problem. Similarly to the 1D Wiener--Hopf method, 
one can formulate the following theorem.  

\begin{theorem}
\label{th:FTth2.1}
Let us consider a function $\phi(x_1 , x_2)$ having its Fourier transform equal to zero outside $Q_1$, and growing no faster than a polynomial for large $x_1, x_2$.
  Then $\mathfrak{F} [\phi](\xi_1 , \xi_2)$ 
can be analytically continued to the domain ${\rm Im}(\xi_1) > 0$, ${\rm Im}(\xi_2) > 0$ and has no singularities there. 
Reciprocally, if a function $\tilde{\Phi}(\xi_1 , \xi_2)$ can be  analytically continued to the 
domain ${\rm Im}(\xi_1) > 0$, ${\rm Im}(\xi_2) > 0$ and has no singularities there. 
Then $\tilde{\Phi}$ is 1/4-based. 
\end{theorem}
%{\bf We should say something smart about continuity of $\phi$ on the real plane. RA, could you, please, check this?} OK, this is addressed further down.
Thus, there is a simple criterion of 1/4-basedness displayed in terms of analyticity. The known problems with building a 2D analog of the Wiener--Hopf method are connected with the concept of 3/4-basedness. Namely, there is no simple criterion of 3/4-basedness. 
An obvious statement that $\tilde{\Phi}(\xi_1 , \xi_2)$ is 3/4-based if and only if it can be represented as 
\[
\tilde{\Phi} (\xi_1, \xi_2) = 
\tilde{\Phi}_1 (-\xi_1, \xi_2) +
\tilde{\Phi}_2 (\xi_1, -\xi_2) +
\tilde{\Phi}_3 (-\xi_1, -\xi_2), 
\]
where $\tilde{\Phi}_1, \tilde{\Phi}_2, \tilde{\Phi}_3$ are 1/4-based functions is not particularly useful on a practical level since it leads to the introduction of three unknown functions instead of one. 

The present paper is an attempt to formulate a criterion of 3/4-basedness. We will show in Section \ref{sec:additivecrossingmain} that it is connected with the idea of additive crossing of branch lines.     
 
%%%%%%%%%%%%%%%%%%%%%%%%%%%%%%%%%%%%%%%%%%%%%%%%%%%%%%%%%%%%%%%%%%%%%%%%%%%%%%%%%%%%%%%

\subsection{Formulation of the functional problem}
\label{sec:functionalproblem}

Let us formulate the functional problem, i.e.\ the problem for the unknown 
function $\tilde W (\xi_1 , \xi_2)$, which would be equivalent to the initial physical problem for 
$u(x_1 , x_2, x_3)$ given in Section \ref{sec:formulationphysical}. We assume that the functional equation (\ref{eq:functionalequation})  is valid, and, using (\ref{eqI0104}), we define $\tilde{U}'$ as follows 
\begin{equation}
\tilde U '(\xi_1, \xi_2) \equiv -i \tilde K (\xi_1 , \xi_2) \, \tilde W (\xi_1 , \xi_2) 
- \frac{1}{(\xi_1 + k_1)(\xi_2 + k_2)} .
\label{eqI0106}
\end{equation}
The functional problem can be formulated in the form of the following theorem. 

\begin{theorem} \label{th:formulationSec2}
Let the function $\tilde W(\xi_1 , \xi_2)$  have the following properties:

\begin{enumerate}

\item
$\tilde W(\xi_1 , \xi_2)$ is 1/4-based.

\item
$\tilde U'(\xi_1, \xi_2)$, as defined by (\ref{eqI0106}), is 3/4-based. 

\item
There exist functions $E_1(\xi_1)$ and $E_2(\xi_2)$ defined for 
complex $\xi_1$ and $\xi_2$, such that  
\begin{equation}
|\tilde W (\xi_1 , \xi_2)| < E_1 (\xi_1) |\xi_2|^{-1/2} 
\qquad \mbox{as} \quad |\xi_2| \rightarrow \infty, \quad {\rm Im}(\xi_2) > 0,
\label{eqI0107}
\end{equation}
\begin{equation}
|\tilde W (\xi_1 , \xi_2)| < E_2 (\xi_2) |\xi_1|^{-1/2} 
\qquad \mbox{as} \quad |\xi_1| \rightarrow \infty, \quad {\rm Im}(\xi_1) > 0 .
\label{eqI0108}
\end{equation}

\item
There exists a function $C(\beta, \psi_1, \psi_2)$ defined for $0<\beta<\pi/2$ and\\ ${0<\psi_{1,2}<\pi}$ such that for real $\Lambda$
\begin{equation}
|\tilde W(\xi_1 , \xi_2)| < C(\beta, \psi_1, \psi_2) \Lambda^{-1-\lambda},
\qquad
\lambda > -1/2,
\label{eqI0109}
\end{equation}
where $\xi_1$ and $\xi_2$ are parametrised as follows for large real $\Lambda$:
\begin{equation}
\xi_1 = \Lambda e^{i \psi_1} \cos(\beta), 
\qquad 
\xi_2 = \Lambda e^{i \psi_2} \sin(\beta), 
\label{eqI0110}
\end{equation}
\end{enumerate}
Then the field $u(x_1 , x_2, x_3)$ defined by 
\begin{equation}
u(x_1 , x_2, x_3) =
-i\mathfrak{F}^{-1}[\tilde K \, \tilde W \, 
\exp \{ i |x_3| / \tilde K \}] (x_1 , x_2) .
\label{eqI0111}
\end{equation}
obeys all conditions of the initial physical problem and is hence a solution to the quarter-plane problem. 
\end{theorem}

\begin{proof}
The Ansatz (\ref{eqI0111}) obeys 
(\ref{eq:Helmholtz}) in $x_3 \ne 0$ by construction. The radiation 
condition is fulfilled because of the imaginary part of $k$ and 
positive imaginary part of $\tilde K^{-1}$. The first condition of the theorem is responsible for the Neumann boundary condition (\ref{eq:BCN}) outside the~QP. 
The second condition yields the Dirichlet boundary condition (\ref{eq:BCDi}) 
on the~QP. The third condition corresponds to the edge conditions (\ref{eqI0102}) and (\ref{eqI0103}). 
Obviously, the functions $E_1$ and $E_2$ are Fourier images of $e_1$ and $e_2$.
Finally, the fourth condition provides the vertex condition\footnote{It is possible to derive a connection formula between $C$ and $c$ but it is quite technical and beyond the purpose of this work} (\ref{eq:vertex}). 
\end{proof}

%{\bf RA, I am not sure that we should derive the connection formula, but we can. A 
%correct formulation of Watson's lemma for 2D is a separate task, an I am not 100 percent sure 
%that our formulation is correct. Anyway, we should not do it in this paper. Maybe we can find it somewhere?} OK I've added a footnote as discussed.

%%%%%%%%%%%%%%%%%%%%%%%%%%%%%%%%%%%%%%%%%%%%%%%%%%%%%%%%%%%%%%%%%%%%%%%%%%%%%%%%%%%%%%

\section{Analytical continuation of $\tilde W$} \label{sec:analyticalcontinuation}

\subsection{Motivation}\label{sec:motivation}

We pay a considerable attention to the possibility of continuing the unknown function to 
some well-defined Riemann manifold. 
Though in the 1D case there is no clear path from the Riemann surface to the actual solution, on the intuitive level, we strongly believe that the possibility to solve the usual 1D Wiener--Hopf problem is linked with the possibility to study the analytical continuation of the 
unknown function {\em without solving the problem}. We illustrate below what we mean by this. 

Consider a sample 1D problem of the form   
\begin{equation}
K(\xi) \, W_+ (\xi) + U_-(\xi) = T(\xi), 
\label{eq:ex01}
\end{equation}
where $\xi$ is a scalar complex variable, $K(\xi)$ is a known coefficient, which is an {\em algebraic function\/}, $W_+$ and $U_-$ are unknown functions analytic in the upper and lower half-plane respectively and $T$ is a known right-hand side, which is a {\em rational function}. Equation (\ref{eq:ex01}) can be rewritten as 
\begin{equation}
 W_+ (\xi)  = K^{-1}(\xi) \,T(\xi) - K^{-1}(\xi) \,U_-(\xi)
\label{eq:ex02}
\end{equation}
and
\begin{equation}
 U_-(\xi) = T(\xi) - K(\xi) \, W_+ (\xi), 
\label{eq:ex03}
\end{equation}
Note that $W_+$ is naturally defined in the upper half-plane, while the right-hand side of 
(\ref{eq:ex02}) is naturally defined in the lower half-plane. Thus, (\ref{eq:ex02}) can be used to continue $W_+$ into the lower half-plane. Similarly, (\ref{eq:ex03}) can be used to continue $U_-$ into the upper half-plane. 

Let $W_+^\star(\xi)$ be the value of $W_+$ at some real $\xi$ of the ``physical sheet'', i.e.\ the value that can be used for computation of some wave field. 
Similarly, let $U_-^\star (\xi)$ be the value of $U_-$ on the physical sheet.

Let us continue $W_+$ along some loop $\sigma$ starting and ending at~$\xi$. The result will be denoted by $W_+(\xi; \sigma)$.
Fix also the physical sheet of the argument of $K$ by $K^\star(\xi)$ and denote different branches of it by $K(\xi ; \sigma)$.

Assume that $\sigma$ is homotopic to a concatenation of loops 
\[
\sigma = \sigma_1^- \sigma_2^+ \sigma_3^- \dots 
\]
where the left loop is passed first, loops $\sigma_j^+$ are located in the upper half-plane, and $\sigma_j^-$ are located in the lower half-plane. See Fig.~\ref{fig:loopsigma} for an illustration.

\begin{figure}[h]
\centering
\includegraphics[width=0.4\textwidth]{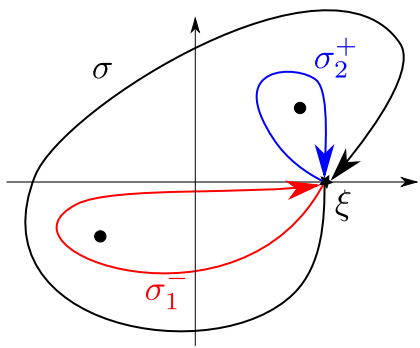}
\caption{Illustration of a loop $\sigma$ homotopic to $\sigma_1^- \sigma_2^+$ in the case of a kernel $K$ with two branch points denoted by $\bullet$}
\label{fig:loopsigma}
\end{figure}
 
It is hence possible to construct $W_+(\xi; \sigma)$ by iterations. First, use (\ref{eq:ex02}) for the loop $\sigma_1^-$ and get 
\begin{equation}
 W_+ (\xi ; \sigma_1^-)  = 
 K^{-1}(\xi ; \sigma_1^-) \,T(\xi) - K^{-1}(\xi ; \sigma_1^-) \,U_-^\star(\xi).
\label{eq:ex04}
\end{equation}
Combining this with (\ref{eq:ex03}) obtain 
\begin{equation}
 W_+ (\xi ; \sigma_1^-)  = 
K^{-1}(\xi ; \sigma_1^-)\, K^\star(\xi) \,W_+^\star(\xi).
\label{eq:ex04a}
\end{equation}
Continuing (\ref{eq:ex04a}) along $\sigma_2^+$, obtain
\begin{equation}
W_+ (\xi ; \sigma_1^- \sigma_2^+)  = K^{-1}(\xi ; \sigma_1^- \sigma_2^+) 
\, K(\xi; \sigma_2^+) \,
W_+^\star(\xi).
\label{eq:ex07}
\end{equation}

This process can be continued, providing $W_+ (\xi ; \sigma_1^- \sigma_2^+ \sigma_3^-)$,
$W_+ (\xi ; \sigma_1^- \sigma_2^+ \sigma_3^- \sigma_4^+)$, etc.
At each step the value  $W_+(\xi ; \sigma)$ is  $W_+^\star(\xi)$ multiplied by some values of 
$K$ and $K^{-1}$ on different branches.

Equations (\ref{eq:ex04}) and (\ref{eq:ex07}) are examples of analytical continuation formulae. 
These formulae provide information about the analytical structure of the 
function $W_+$ even without solving the Wiener--Hopf equation. Namely, one can reveal the structure of the Riemann surface of $W_+$ and find
all singularities on this Riemann surface.

The aim of this section is to build formulae of analytical continuation for the 
2D Wiener--Hopf problem formulated above. Unfortunately, in the 2D case such formulae 
are considerably more complicated, namely, they include integral operators. Still, they are helpful and provide an important information about 3/4-based functions.

\subsection{Primary formulae for analytical continuation}\label{sec:formulae}

%The  domains of analyticity of 
%two unknown functions 
%should be intersecting and they should cover the whole complex space 
%$\mathbb{C} \times \mathbb{C}$. 

The two unknown functions are, formally, $\tilde W$ and $\tilde U'$. As we noted above, the natural domain of analyticity of the 1/4-based function $\tilde W (\xi_1 , \xi_2)$ is ${\rm Im}(\xi_1) > 0$, ${\rm Im}(\xi_2) > 0$. 
In our consideration we use an important ``physical'' conjecture that the field has 
an exponential decay due to the presence of the imaginary part of $k$ corresponding to 
absorption in the medium.  
Namely, we estimate the field as 
\begin{equation} 
|u(x_1, x_2 , 0^+) | < C_{\epsilon} \exp\{- 4\kappa \sqrt{x_1^2 + x_2^2}\}
\quad
\mbox{ as } 
\quad
\sqrt{x_1^2 + x_2^2} \to \infty
\label{eq:estimation}
\end{equation}
for 
\[
\kappa = \tfrac{1}{4} \, {\rm min}\, (\epsilon/2, {\rm Im}(k_1), {\rm Im}(k_2)). 
\]
Thus, since $2\sqrt{x_1^2+x_2^2}>x_1+x_2$, one can easily prove that $\tilde W(\xi_1 , \xi_2)$ is actually analytic in the 
wider domain 
\[
{\rm Im}(\xi_1) > -2 \kappa ,  \qquad {\rm Im}(\xi_2) > -2 \kappa.
\]
%{\bf RA, can we add a solid reference here?}
Note that nothing can be said {\em a priori\/} about the domain 
of analyticity of the second unknown function $\tilde U' (\xi_1 , \xi_2)$, which is 
3/4-based. 

%Note that due to the very simple functional equation (\ref{eq:functionalequation}) 
%one can study the analyticity domain of $\tilde W$ instead of 
%$\tilde U'$, admitting the presence of branch lines or poles and describing the 
%character of branching at these lines. 

%We propose here to use the {\em analytical continuation formulae\/} to extend the 
%domain of analyticity of $\tilde W$. There are two such formulae given by Theorem 
%\ref{th:th3} below. 

Let us now introduce the important function $\gamma$ by 
\begin{equation}
\gamma(\xi_1 , \xi_2) \equiv \sqrt{\sqrt{k^2 - \xi_1^2} + \xi_2},
\label{eqI0112}
\end{equation}
for which the ``arithmetic'' branches of the square roots on the ``physical sheet'' are considered, i.e. the value of 
a square root for a real positive argument is positive real. 
Note that $\gamma$ participates in a multiplicative factorisation 
of $\tilde K$:
\begin{equation}
\tilde K (\xi_1 , \xi_2) = \frac{1}{\gamma(\xi_1 , \xi_2) \, \gamma(\xi_1 , -\xi_2)} = 
\frac{1}{\gamma(\xi_2 , \xi_1) \, \gamma(\xi_2 , -\xi_1)} .
\label{eqI0113}
\end{equation}
Using this function, we can formulate our first analytical continuation formulae.
\begin{theorem}\label{th:th3}
For $|\textup{Im}(\xi_{1,2})|<\kappa$, $\tilde W$ obeys the following relations: 
\begin{eqnarray}
\tilde W(\xi_1 , \xi_2) &=& 
\frac{ \gamma(\xi_1, \xi_2)}{4\pi^2}\!\!\!\!\!
\int \limits_{-\infty - i \kappa}^{\infty - i \kappa}\!\!\!\!\!\!\mathd \xi_2'  
\int \limits_{-\infty + i \kappa}^{\infty + i \kappa}\!\!\!\!\!\!\mathd \xi_1' 
\,\frac{
\gamma(\xi_1 , -\xi_2') \, \tilde K (\xi_1' , \xi_2')\, \tilde W(\xi_1' , \xi_2')}{
(\xi_1' - \xi_1) (\xi_2' - \xi_2)} \nonumber \\ 
&+&
\frac{i \gamma(\xi_1, \xi_2) \, \gamma(\xi_1, k_2)}{(\xi_1 + k_1)(\xi_2 + k_2)}
\label{eq:stupidformula1} \\
\tilde W(\xi_1 , \xi_2) &=&
\frac{ \gamma(\xi_2, \xi_1)}{4\pi^2}\!\!\!\!\!
\int \limits_{-\infty - i \kappa}^{\infty - i \kappa}  
\!\!\!\!\!\!\mathd \xi_1'
\int \limits_{-\infty + i \kappa}^{\infty + i \kappa} 
\!\!\!\!\!\!\mathd \xi_2' 
\,\frac{
\gamma(\xi_2 , -\xi_1') \, \tilde K (\xi_1' , \xi_2')\, \tilde W(\xi_1' , \xi_2')}{
(\xi_1' - \xi_1) (\xi_2' - \xi_2)} \nonumber \\ 
&+&
\frac{i \gamma(\xi_2, \xi_1) \, \gamma(\xi_2, k_1)}{(\xi_1 + k_1)(\xi_2 + k_2)}
. 
\label{eq:stupidformula2}
\end{eqnarray}
\end{theorem}
%{\bf RA, most probably, you already have this fragment written. Please, check the formula and insert the proof. We may put the proof into the appendix as well.} OK Proof is in the appendix.
The proof of this Theorem \ref{th:th3} is given in Appendix \ref{app:proofth3}. The relations (\ref{eq:stupidformula1}) and (\ref{eq:stupidformula2}) are integral equations for $\tilde W$, however, we will use them in another way, namely, for continuation of $\tilde W$ onto a Riemann manifold. Indeed, even if the right hand side of formulae (\ref{eq:stupidformula1}) and (\ref{eq:stupidformula2})
require $\tilde W$ to be known in a narrow strip near the real plane, one can choose $\xi_1 , \xi_2$ (the arguments of $\tilde W$ in the left hand side) to belong to a much wider domain, and, doing so, provide the required continuation. This will be explained in more detail below. 

%%%%%%%%%%%%%%%%%%%%%%%%%%%%%%%%%%%%%%%%%%%%%%%%%%%%%%%%%%%%%%%%%%%%%%%%%%%%%%%%%%%%%%%

\subsection{Domains for analytical continuation}

Let us define the domains $H^+$ and $H^-$ as domains of a complex plane 
of a single variable, being  the upper and the lower half-planes cut along the curves $h^+$ and $h^-$ (see Fig.~\ref{fig02}, left). These curves are the sets of points 
\[
h^{\pm}: \quad \xi = \pm \sqrt{k^2 - \tau^2}, \quad \tau \in \mathbb{R}.  
\]  
Let $\hat H^+$ be the upper half-plane, which is not cut along $h^+$. 
The domains in all cases are open, i.e.\ the boundary is not included. 
The boundary of $H^-$ consists of two parts: 
the real axis and the curve $h^-$ (passed two times, from $-i \infty$ to $-k$ along the right shore of the cut and backwards along the left shore of the cut).
Denote this pass by $P$ (see Fig.~\ref{fig02}, right), so the boundary of $H^-$ is 
$\partial H^-=\mathbb{R}\cup P$.  Similarly, the 
boundary of $H^+$ is $\partial H^+=\mathbb{R}\cup(-P)$. 

%%%%%%%%%%%%%%%%%%%%%%%%%%%%
 \begin{figure}[ht]
 \centerline{\epsfig{file=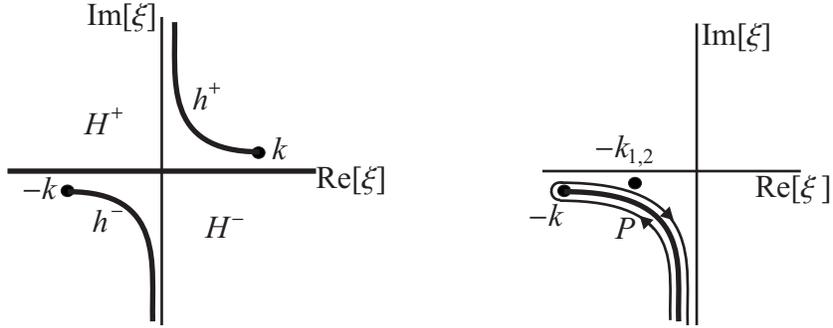}}
 \caption{Domains $H^{\pm}$ (left), contour $P$, right}
 \label{fig02}
 \end{figure}
%%%%%%%%%%%%%%%%%%%%%%%%%%%%

An important property of the domains $H^\pm$ is the following lemma.
 
\begin{lemma}\label{lem:kappasign}
Let the branch of the square root $\sqrt{k^2 - \xi^2}$ be chosen such that this function have positive imaginary part for real $\xi$. If $\xi \in H^+$ or $\xi \in H^-$, then $\sqrt{k^2 - \xi^2} \in H^+$.  
\end{lemma}

\begin{proof}
Let be $\xi \in H^+$. 
Consider the mapping $\xi \to \sqrt{k^2 - \xi^2}$. This mappings maps the real axis onto 
$-P$, and $-P$ onto the real axis. Thus, according to the principles of conforming mapping, $H^+$ is mapped  onto $H^+$. 

Let now be $\xi \in H^-$. The real axis maps onto $-P$, and $P$ maps onto the real axis. Thus, $H^-$ is mapped onto $H^+$. 
\end{proof}

%%%%%%%%%%%%%%%%%%%%%%%%%%%%%%%%%%%%%%%%%%%%%%%%%%%%%%%%%%%%%%%%%%%%%%%%%%%%%%%    

\subsection{First step of analytical continuation}

Initially, as a 1/4-based function, $\tilde W (\xi_1 , \xi_2)$ is analytic in the 
domain $\hat H^+ \times \hat H^+$.
It is easy to show that, besides,  $\tilde W$ is continuous at the boundary of this domain
(for example one can note, as it has been mentioned above, that $\tilde W$ is 
analytic in a slightly wider domain).

%We can continue $\tilde W$ onto some Riemann 
%manifold. Namely, $\tilde W$ should be continued along any contour in 
%$\mathbb{C} \times \mathbb{C}$ except some branch and polar 2-lines.  

%We have a conjecture that successive 
%application of formulae (\ref{eq:stupidformula1}) and (\ref{eq:stupidformula2}) leads to %an analytical continuation of $\tilde W$ outside of its natural domain of analyticity. We %are not going to perform an infinite number of steps (for this, application of Pham's %theory is needed \cite{Pham}). Instead, we perform only two steps. 
%This will enable us to introduce the property of additive crossing of branch lines, which %is closely connected with the 3/4-basedness. 

We are going to apply our formulae for analytical continuation twice. 
Each time the domain of analyticity of $\tilde W$ will be extended. 
In this subsection we apply the formulae of analytical continuation 
for the first time. 

Our first aim is to continue $\tilde W$ into the domain $(H^- \setminus \{-k_1\} ) \times \hat H^+$. 
Formula (\ref{eq:stupidformula1}) provides an analytical continuation of 
$\tilde W(\xi_1 , \xi_2)$
into a narrow strip surrounding the real plane, e.g.\ into the domain 
$|{\rm Im}(\xi_1)| < \kappa $, $|{\rm Im}(\xi_2)| < \kappa $. 
Then, fix $(\xi_1, \xi_2)$ belonging to this strip such that   
${\rm Im}(\xi_1) < 0$, ${\rm Im}(\xi_2) > 0$ and change the integration surface\footnote{The notations here and below should be clear. We have in mind 
double contour integrals. The first factor relates everywhere to $\xi_1'$, the 
second one to $\xi_2'$. The left end of the interval is the start of the contour. In this paper we avoid using the standard differential form notations.} in 
(\ref{eq:stupidformula1}) from 
\[
(-\infty + i \kappa , \infty + i \kappa) \times (-\infty - i \kappa , \infty - i \kappa)
\qquad
\mbox{to}   
\qquad
(-\infty  , \infty ) \times (-\infty  , \infty ),
\]
as illustrated in Fig.~\ref{fig:deform30to33}.

\begin{figure}[ht]
\centering
\includegraphics[width=0.7\textwidth]{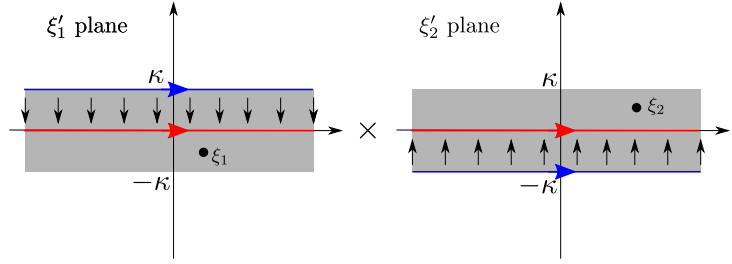}
\caption{Illustration of the contour deformation when passing from (\ref{eq:stupidformula1}) to (\ref{eq:stupidformula1_prime}) }
\label{fig:deform30to33}
\end{figure}

This deformation of the integration surface does not change the value of the integral 
due to the 2D analog of Cauchy's theorem or Stokes' theorem \cite{Shabat2}. 
Thus, for ${\rm Im}(\xi_1) < 0$, ${\rm Im}(\xi_2) > 0$ the continuation is given by a slightly modified formula instead of (\ref{eq:stupidformula1}): 
\begin{eqnarray}
\tilde W(\xi_1 , \xi_2) &=& 
\frac{ \gamma(\xi_1, \xi_2)}{4\pi^2}
\int \limits_{-\infty }^{\infty }\mathd \xi_2'  
\int \limits_{-\infty }^{\infty }\mathd \xi_1' 
\,\frac{
\gamma(\xi_1 , -\xi_2') \, \tilde K (\xi_1' , \xi_2')\, \tilde W(\xi_1' , \xi_2')}{
(\xi_1' - \xi_1) (\xi_2' - \xi_2)} \nonumber \\
&+& \frac{i \gamma(\xi_1, \xi_2) \, \gamma(\xi_1, k_2)}{(\xi_1 + k_1)(\xi_2 + k_2)}.
\label{eq:stupidformula1_prime}
\end{eqnarray}
Using this formula, we can continue $\tilde W$ into the domain $(H^- \setminus \{-k_1\} ) \times \hat H^+$:

\begin{theorem}\label{th:th1step1}
The function $\tilde W (\xi_1 , \xi_2)$ is analytic in the domain 
${(H^- \setminus \{-k_1\} ) \times \hat H^+}$.  
In the vicinity of $\{-k_1\} \times \hat H^+$ the function can be represented as 
\begin{equation}
\tilde W(\xi_1 , \xi_2) = \frac{1}{\xi_1 + k_1} 
\frac{i\, \gamma(k_1, \xi_2) \gamma(k_1 , k_2)}{\xi_2 + k_2} + O((\xi_1 + k_1)^0)
\label{eqI0114}
\end{equation}
The function is analytic on the boundary elements $\mathbb{R} \times \hat H^+$, 
$P \times \hat H^+$, $H^- \times \mathbb{R}$, and continuous on the 
boundary element $P \times \mathbb{R}$.
\end{theorem}

\begin{remark} 
One can see that we are going to prove analyticity in the domain
$B_1 \times B_2$ (of real dimension 4), where $B_1 = H^-$, $B_2 = \hat H^+$ (minus the polar set), 
analyticity at the points of the boundary 
$(\ptl B_1 \times B_2) \cup (B_1 \times \ptl B_2)$ (of real dimension 3) and continuity
at the points of the skeleton of the boundary 
$(\ptl B_1 \times \ptl B_2)$ (of real dimension 2).   

Analyticity at any point of $(\ptl B_1 \times B_2) \cup (B_1 \times \ptl B_2)$
has the sense that
the function can 
be analytically continued to a small polydisc with the center at such point.

Continuity of the function at the points belonging to 
$(\ptl B_1 \times \ptl B_2)$
 means that 
 the function tends to a certain common
limit while the argument tends to the point along any continuous 
path going in the domain
$(B_1 \times B_2) \cup (\ptl B_1 \times B_2) \cup (B_1 \times \ptl B_2)$.
\end{remark}  

\begin{proof}
Let us start by analysing the integral in (\ref{eq:stupidformula1_prime}):
\[
J(\xi_1, \xi_2) \equiv 
\int \limits_{-\infty }^{\infty }\mathd \xi_2'  
\int \limits_{-\infty }^{\infty }\mathd \xi_1' 
\,\frac{
\gamma(\xi_1 , -\xi_2') \, \tilde K (\xi_1' , \xi_2')\, \tilde W(\xi_1' , \xi_2')}{
(\xi_1' - \xi_1) (\xi_2' - \xi_2)}.
\]

\paragraph*{Domain of analyticity of $J$.} First of all, the integral converges due to the growth conditions (\ref{eqI0107})--(\ref{eqI0109}) of the functional problem. Let us now study all factors of the integrand and make sure that neither of them is 
singular for $\xi_1 \in H^-\setminus \{-k_1\}$ and $\xi_2 \in \hat H^+$.  

The functions $\tilde K(\xi_1' , \xi_2')$ and $\tilde W(\xi_1' , \xi_2')$ do not depend 
on $\xi_1, \xi_2$ and do not pose any problem. 

The polar factors $(\xi_1' - \xi_1)^{-1}$ and $(\xi_2' - \xi_2)^{-1}$ are regular
since $\xi_{1,2}$ are not real, while $\xi_{1,2}'$ are. 

The factor $\gamma(\xi_1 , - \xi_2') = \sqrt{\sqrt{k^2 - \xi_1^2} - \xi_2'}$  would have branch points for $\xi_1=\pm k$ and $\sqrt{k^2-\xi_1^2}=\xi_2'$. However $\pm k \notin H^-$ by definition of $H^-$ and, since $\xi_1\in H^-$, we know by Lemma \ref{lem:kappasign} that $\sqrt{k^2-\xi_1^2}$ belongs to $H^+$ which does not contain the real number $\xi_2'$. 

Therefore, $J$ is convergent and its integrand is analytic in ${(H^- \setminus \{-k_1\} ) \times \hat H^+}$. Hence, since Morera's theorem holds in two complex variables (see \cite{Shabat2}), $J$ is analytic in this domain.

\paragraph*{Analyticity of $J$ on the boundary.}
%Since the integrand is continuous for the values of parameters
%$(\xi_1, \xi_2) \in P \times \hat H^+$ and the integral converges there, 
%the integral $J$ is 
%continuous on this boundary fragment.{\color{red}Is this last sentence necessary?}

The fact that $J(\xi_1 , \xi_2)$ is analytic at the points 
of $\mathbb{R} \times \hat H^+$,
$H^- \times \mathbb{R}$ 
and $P \times \hat H^+$
is supported by the possibility to use Cauchy's theorem 
and change the integration surface
to a product $\Theta_1 \times \Theta_2$, where contours
$\Theta_1$ and $\Theta_2$ are contained within the strip $|{\rm Im}(\xi)| < \kappa$. 

Let us consider a point $(\xi_1^\star,\xi_2^\star)\in\mathbb{R} \times \hat{H}^+$ (Step 1). First of all, the term $\gamma(\xi_1,-\xi_2')$ is analytic at $\xi_1^\star$ since $\xi_1^\star\notin h^{\pm}$. The term that is actually problematic is the polar factor $(\xi_1-\xi_1')^{-1}$. %As illustrated in Fig. \ref{fig:proof32_1} for each point $(\xi_1^\star,\xi_2^\star)\in\mathbb{R} \times \hat{H}^+$ one should use the surface $\Theta_1=\mathbb{R}$ and $\Theta_2=(-\infty + i \kappa' , \infty + i \kappa')$. 
To avoid this problem, first consider $(\xi_1,\xi_2^\star)$, with $\xi_1\in H^-$ (Step 2). From what we saw above, $J$ is clearly analytic there. Moreover its integrand (as a function of $(\xi_1',\xi_2')$) is analytic within the strip $|\text{Im}(\xi_1')|<\kappa$. Hence, as illustrated in Fig.~\ref{fig:proof32_1}, we can deform the contour up from $\mathbb{R}$ to $\Theta_1=(-\infty + i \kappa' , \infty + i \kappa')$ without changing the value of $J$ (Step 3). Now we can safely let $\xi_1$ move towards $\xi_1^\star$ from within $H^-$, without hitting the singularity of the polar factor  $(\xi_1-\xi_1')^{-1}$. Hence $J$ can be analytically continued on $\mathbb{R}\times \hat{H}^+$.
%\[
%(-\infty + i \kappa' , \infty + i \kappa') \times
%(-\infty , \infty )
%\]
%for a sufficiently small $\kappa'$ depending on $\xi_2$.

\begin{figure}[ht]
\centering
\includegraphics[width=0.7\textwidth]{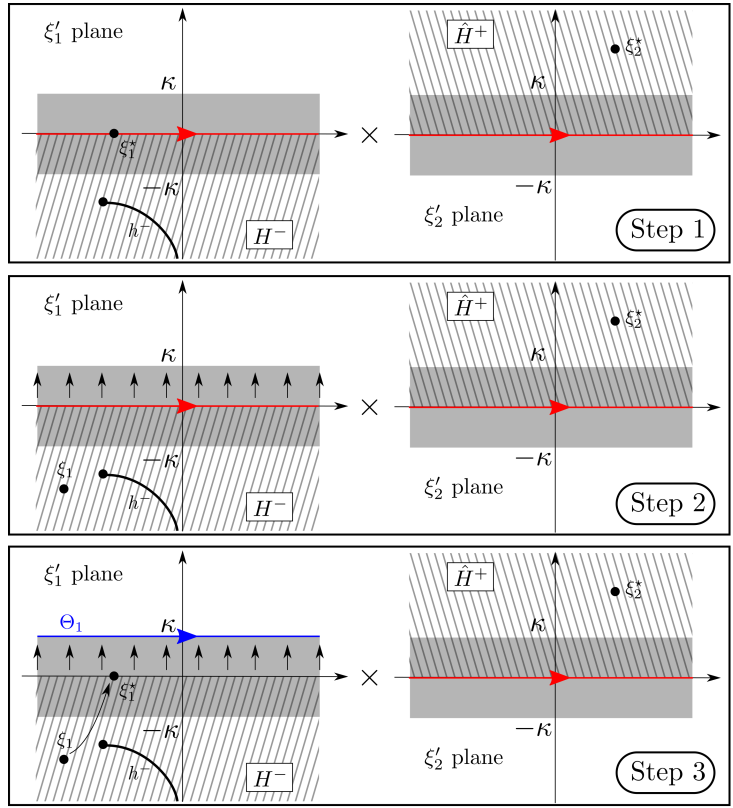}
\caption{Illustration of the contour deformation needed to prove analyticity on $\mathbb{R} \times \hat{H}^+$ }
\label{fig:proof32_1}
\end{figure}

Let us now consider a point $(\xi_1^\star,\xi_2^\star)\in H^- \times \mathbb{R}$ (Step 1). This time the problematic term is the polar factor $(\xi_2-\xi_2')^{-1}$, which may become singular at $\xi_2^\star$. Hence, we want to deform the contour down in the $\xi_2'$ plane, but this time we need to be a little bit more careful. Consider a point $(\xi_1^\star,\xi_2)$, with $\xi_2\in \hat{H}^+$ (Step 2). Note that at this stage, the singular loci of the function $\gamma(\xi_1, -\xi_2')$ is $h^-$. Now, as illustrated in Fig.~\ref{fig:proof32_2}, when deforming the contour down to a new contour $\Theta_2$, this singular loci becomes a curve surrounding $h^-$. Hence when deforming the contour down, one should be careful that this loci does not intersect the point $\xi_1^\star$ (Step 3). Note that since $\xi_1^\star\notin h^-$, it is always possible to find a $\kappa'>0$ small enough such that this isn't the case. Let us choose $\Theta_2=(-\infty - i \kappa', \infty - i \kappa')$ so that this deformation does not affect the value of $J(\xi_1^\star,\xi_2)$. We can then safely let $\xi_2$ move towards $\xi_2^\star$ from within $\hat{H}^+$, without hitting the singularity of the polar factor  $(\xi_2-\xi_2')^{-1}$. Hence $J$ can be analytically continued on $H^-\times\mathbb{R}$.

\begin{figure}[ht]
\centering
\includegraphics[width=0.7\textwidth]{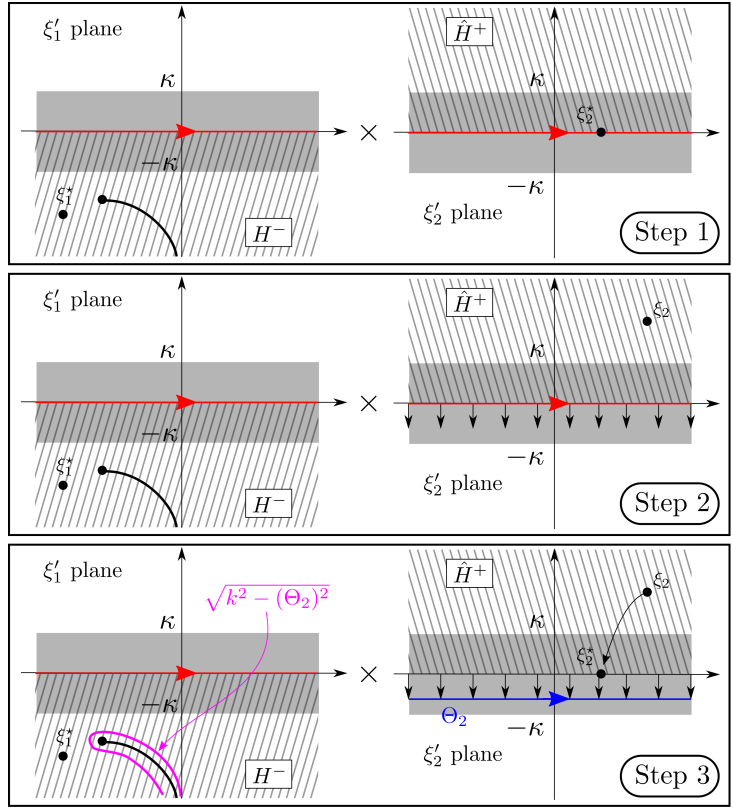}
\caption{Illustration of the contour deformation needed to prove analyticity on $\mathbb{H^-} \times \mathbb{R}$ }
\label{fig:proof32_2}
\end{figure}

Finally, let us consider a point $(\xi_1^\star,\xi_2^\star)\in P\times \hat{H}^+$ (Step 1). Here the polar factors do not present any issues, however, the function $\gamma(\xi_1,-\xi_2')$ does. Its singular loci in the $\xi_1$ plane for $\xi_2'\in\mathbb{R}$ is $P$, while its singular loci in the $\xi_2'$ plane for $\xi_1\in P$ is $\mathbb{R}$. Let us further assume that $\xi_1^\star$ is on the right shore of $P$ as illustrated in Fig.~\ref{fig:proof32_3} (the left shore case is very similar). The problematic point in the $\xi_2'$ plane is $\sqrt{k^2-(\xi_1^\star)^2}$ which is a branch point of the function $\gamma(\xi_1^\star,-\xi_2')$ and also belongs to the contour of integration. As before, we will endeavour to deform the contour to avoid this problem. Let us consider $\xi_1\in H^-$ (Step 2).  Now, deform the contour to a contour $\Theta_2$ indented below $\sqrt{k^2-(\xi_1^\star)^2}$ and above $-\sqrt{k^2-(\xi_1^\star)^2}$ (Step 3). As illustrated in Fig.~\ref{fig:proof32_3}, this has for effect to deflect the singular loci in the $\xi_1$ plane away from $\xi_1^\star$, while ensuring that its two shores do not cross in the process. The value of $J(\xi_1,\xi_2^\star)$ remain unchanged by such deformation. It is now possible to let $\xi_1$ approach $\xi_1^\star$ freely from within $H^-$ without hitting any singularity, and so $J$ can be analytically continued to $(\xi_1^\star,\xi_2^\star)$.

\begin{figure}[ht]
\centering
\includegraphics[width=0.7\textwidth]{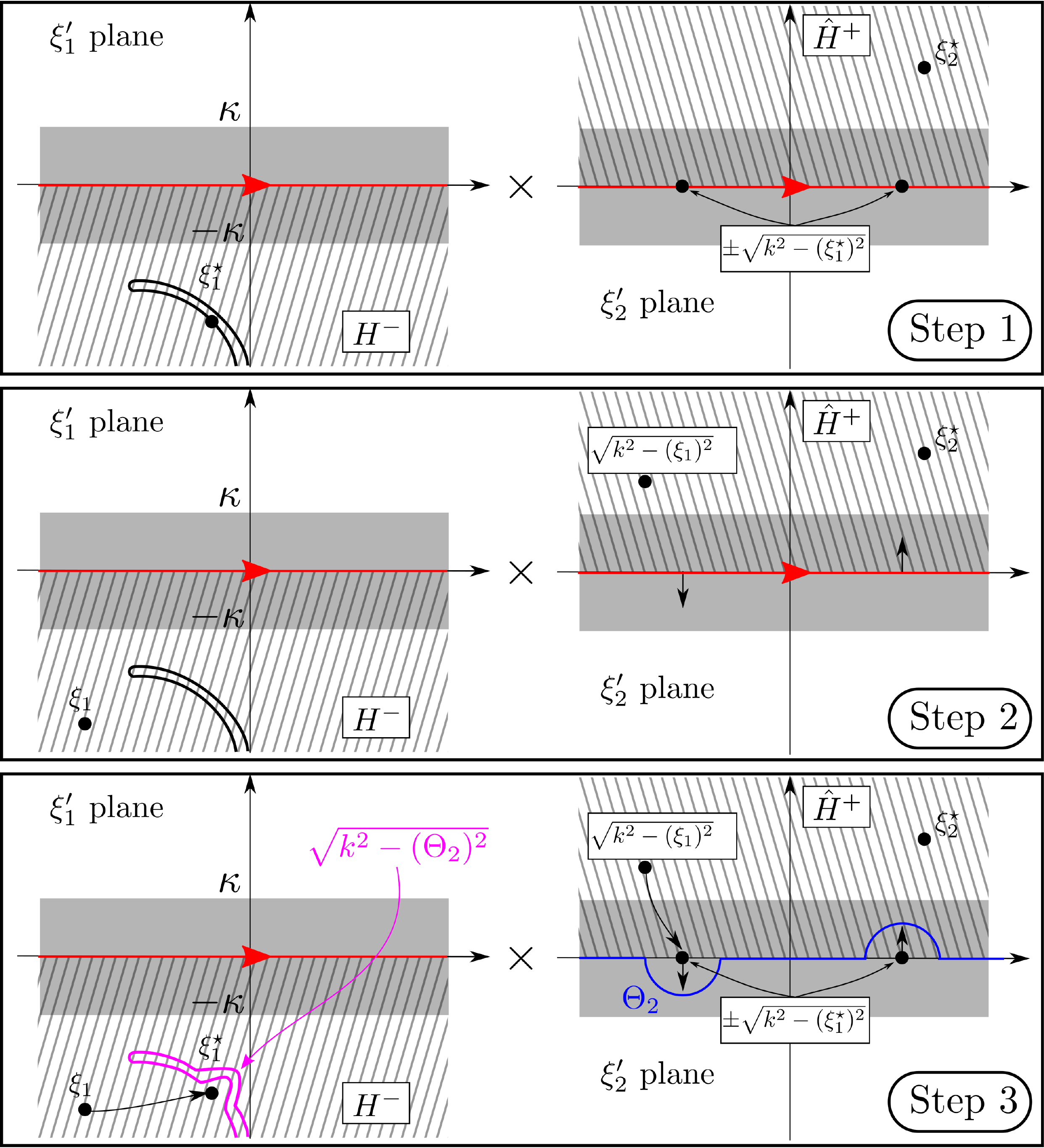}
\caption{Illustration of the contour deformation needed to prove analyticity on $P \times \hat{H}^+$ }
\label{fig:proof32_3}
\end{figure}
  
%{\bf RA, should we give more details?}

\paragraph*{Continuity of $J$ on the skeleton.} We shall now focus on proving the continuity of $J(\xi_1,\xi_2)$ at a point $(\xi_1^\star,\xi_2^\star)\in P \times \mathbb{R}$, when this point is approached form within $H^- \times \hat{H}^+$ possibly including parts of the boundary considered above. In order to do so, decompose $\tilde W$ as $\tilde W = \tilde W_1  + \tilde W_2$, where
\begin{eqnarray*}
\tilde W_1(\xi_1' , \xi_2'; \xi_2) =  
\tilde W (\xi_1' , \xi_2') - \tilde W (\xi_1' , {\rm Re}(\xi_2)) 
& \text{and} & 
\tilde W_2(\xi_1' , \xi_2';\xi_2) =  
\tilde W (\xi_1' , {\rm Re}(\xi_2)). 
\end{eqnarray*}
This naturally decomposes $J$ into $J_1+J_2$.
The first term is continuous as $\xi_2 \rightarrow \xi_2^\star$, since the factor $(\xi_2' - \xi_2)$ is now a removable singularity and so $J_1$ is continuous. The second term does have a polar singularity, though all the other terms are well behaved, and we can calculate the $\xi_2'$ integral of $J_2$ using a residue that behaves like 
\[
\frac{\gamma(\xi_1,-\xi_2^\star)}{\xi_1'-\xi_1} \tilde{K}(\xi_1',\xi_2^\star) \tilde{W}(\xi_1',\xi_2^\star),
\]
which is continuous as $\xi_1\rightarrow \xi_1^\star$. Hence, the integral possesses the required continuity. 

\paragraph*{External factor and additive term.} The factor $\gamma(\xi_1 , \xi_2)$ in front of the integral is analytic by Lemma \ref{lem:kappasign}. Indeed, since $\xi_1\in H^-$, $\sqrt{k^2-\xi_1^2}$ belongs to $H^+$ and cannot be equal to $-\xi_2$ since $\xi_2 \in \hat{H}^+$. The analyticity and the continuity on the boundary can be established in a straightforward way. The additive term can also be analysed directly, and it fits the theorem. In particular, the only singularity of $\tilde W$ in the domain $H^-  \times \hat H^+$ is due to this term. 
\end{proof}

Using the modification of (\ref{eq:stupidformula2}) written as 
\begin{equation}
\tilde W(\xi_1 , \xi_2) =
\frac{ \gamma(\xi_2, \xi_1)}{4\pi^2}
\int \limits_{-\infty }^{\infty }  
\mathd \xi_1'
\int \limits_{-\infty }^{\infty } 
\mathd \xi_2' 
\,\frac{
\gamma(\xi_2 , -\xi_1') \, \tilde K (\xi_1' , \xi_2')\, \tilde W(\xi_1' , \xi_2')}{
(\xi_1' - \xi_1) (\xi_2' - \xi_2)}  
\label{eq:stupidformula2_prime}
\end{equation}
\[
\qquad \qquad \qquad \qquad \qquad \qquad
+\frac{i \gamma(\xi_2, \xi_1) \, \gamma(\xi_2, k_1)}{(\xi_1 + k_1)(\xi_2 + k_2)},
\]
valid for ${\rm Im}(\xi_1) > 0$, ${\rm Im}(\xi_2) < 0$, 
one can prove the ``symmetrical'' theorem:

\begin{theorem}\label{th:th2step1}
The function $\tilde W (\xi_1 , \xi_2)$ is analytic in the domain 
$\hat H^+ \times (H^- \setminus \{-k_2\} )$.  
In the vicinity of $ \hat H^+ \times \{-k_2\}$ the function can be represented as 
\begin{equation}
\tilde W(\xi_1 , \xi_2) = \frac{1}{\xi_2 + k_2} 
\frac{i\, \gamma(k_2, \xi_1) \gamma(k_2 , k_1)}{\xi_1 + k_1} + O((\xi_2 + k_2)^0).
\label{eqI0114a}
\end{equation}
The function is analytic on the boundary elements $\hat H^+ \times \mathbb{R}$, 
$\hat H^+ \times  P$, $\mathbb{R} \times H^-$, 
and continuous on the boundary element
$\mathbb{R} \times P$.
\end{theorem}

%%%%%%%%%%%%%%%%%%%%%%%%%%%%%%%%%%%%%%%%%%%%%%%%%%

\subsection{Second step of analytical continuation}

The second step of our analytical continuation will be based on a deformation of the
integration surface in (\ref{eq:stupidformula1_prime})
into a product $\mathbb{R} \times P$, resulting in the following theorem:

%%%%%%%%%%%%%%%%%%%%%%%%%%%%
% \begin{figure}[ht]
% \centerline{\epsfig{file=newcontour.eps}}
% \caption{Contour $P$}
% \label{fig03}
% \end{figure}
%%%%%%%%%%%%%%%%%%%%%%%%%%%%

\begin{theorem}\label{th:th34}
The function $\tilde W$ obeys the following relation: 
\begin{eqnarray}
\tilde W(\xi_1 , \xi_2) & = & 
\frac{ \gamma(\xi_1, \xi_2)}{4\pi^2}
\int \limits_P  
\mathd \xi_2'
\int \limits_{-\infty}^{\infty} 
\mathd \xi_1' 
\frac{
\gamma(\xi_1 , -\xi_2') \, \tilde K (\xi_1' , \xi_2')\, \tilde W(\xi_1' , \xi_2')}{
(\xi_1' - \xi_1) (\xi_2' - \xi_2)} \nonumber \\ 
& + & 
\frac{i \gamma(\xi_1, \xi_2) \, \gamma(\xi_1, k_2) \, \gamma(k_2, k_1) }{
(\xi_1 + k_1)\, (\xi_2 + k_2) \, \gamma(k_2 , -\xi_1)}
,
\label{eq:stupidformula11}
\end{eqnarray}
where the left-hand side is defined in the domain 
$-\kappa < {\rm Im}(\xi_1) < 0$, 
$0 < {\rm Im}(\xi_2) < \kappa$.
\end{theorem}

\noindent
\begin{remark}
The unknown function $\tilde W (\xi_1' , \xi_2')$ in the right-hand side 
is defined by Theorem~\ref{th:th1step1}. Namely, the values on $\mathbb{R} \times P$
are defined by continuity from the values defined by formula 
(\ref{eq:stupidformula1_prime}). 
This is not used, but for most of the points of $\mathbb{R} \times P$
(for all non-singular points) the values of $\tilde W$
can be found from integral of the form 
(\ref{eq:stupidformula1_prime}) with appropriate integration surfaces. 

The factor $\tilde K(\xi_1' , \xi_2')$ is singular at the points
of $\mathbb{R} \times P$  
where $\xi_1'^2 + \xi_2'^2 = k^2$, however, they produce an integrable singularity. 
The branch of $\tilde K(\xi_1' , \xi_2')$ is chosen by continuity 
(the choice is defined by physical reasons on $\mathbb{R} \times \mathbb{R}$, 
then, by continuity, on $\mathbb{R} \times P$). 
Thus the integral (\ref{eq:stupidformula11}) can be considered as an improper integral.
\end{remark}

\begin{proof}
Let us modify the continuation formula (\ref{eq:stupidformula1_prime}) as follows.  
Swap the order of integration, fix $\xi_1' \in \mathbb{R}$ and
deform the contour of integration in $\xi_2'$ from the real axis to~$P$.
Note that the function $\tilde W$ is analytic between the new and the old 
contour (apart from a pole at $\xi_2'=-k_2$, which will be taken into account later)
according to Theorem~\ref{th:th2step1}.

While the contour in $\xi_2'$ is deformed, it hits no singularities of the factor 
$\tilde K(\xi_1' , \xi_2')$, since for real $\xi_1'$ the singularities for $\xi_2'$
are located only on  $P$ and $-P$. 
The contour also does not hit singularities of $\gamma(\xi_1, -\xi_2')$ since, by Lemma~\ref{lem:kappasign}, $\sqrt{k^2 - \xi_1^2}$ belongs to $H^+$ for $-\kappa<{\rm Im}(\xi_1) <0$, and 
$\xi_2'$ belongs to~$H^-$.
Finally, the contour does not hit the singularity of the polar factor 
$(\xi_2' - \xi_2)^{-1}$ for an obvious reason.

Thus, the contour deformation (taking into account an additional loop around the pole 
$-k_2$) obeys the condition of 1D Cauchy's theorem, and it does not change the integral.  

When such a deformation is made, the contour hits this pole at $\xi_2' = - k_2$
(and no other singularity of the integrand). 
The residue of the integrand at that point can be obtained by Theorem \ref{th:th2step1}, and is
\[
\frac{-i\gamma(\xi_1,k_2) \tilde K(\xi_1',-k_2) \gamma(k_2,\xi_1') \gamma(k_2,k_1)}{(k_2+\xi_2)(\xi_1'+k_1)(\xi_1'-\xi_1)}.
\] 
It is then necessary to integrate $-2i\pi$ times this residue over $\xi_1'$. It turns out that the integral is just a Cauchy sum-split integral of a function that only has a pole at $-k_1$ in the $\xi_1'$ plane. The split can hence be performed explicitly by the pole removal technique. This leads to two terms as follows
\[ 
\frac{i \gamma(\xi_1, \xi_2) \, \gamma(\xi_1, k_2) \, \gamma(k_2, k_1) }{
(\xi_1 + k_1)\, (\xi_2 + k_2) \, \gamma(k_2 , -\xi_1)}
-
\frac{i \gamma(\xi_1, \xi_2) \, \gamma(\xi_1, k_2) }{
(\xi_1 + k_1)\, (\xi_2 + k_2) },
\]
where the right part of the equality (\ref{eqI0113}) has been used. The second term cancels with the second term of (\ref{eq:stupidformula1}) and the theorem is proved.
\end{proof}

The domain of validity of formula (\ref{eq:stupidformula11}) given in Theorem \ref{th:th34} intersects with the domain of natural analyticity 
of $\tilde W$, i.e.\ ${\rm Im}(\kappa_1) > -2\kappa $, ${\rm Im}(\kappa_2) > -2\kappa$. 
Thus, formula (\ref{eq:stupidformula11}) can provide an analytical continuation 
of $\tilde W$. 
Theorems~\ref{th:th1step1} and~\ref{th:th2step1} perform 
a continuation into the domains 
 ${\rm Im}(\xi_1) \le 0$,  ${\rm Im}(\xi_2) \ge 0$ and 
 ${\rm Im}(\xi_1) \ge 0$,  ${\rm Im}(\xi_2) \le 0$ with some cuts. 
Here our aim is to continue this function into the domain 
${\rm Im}(\xi_1) \le 0$,  ${\rm Im}(\xi_2) \le 0$
(also with some cuts). 

We find that it is convenient to study a continuation of $-i \tilde K \, \tilde W$ instead of $\tilde W$. 
Indeed, these functions are the same up to a factor known explicitly. The required continuation is obtained from the following theorem. 

\begin{theorem}\label{th:th7}
The function $\tilde U$ defined by 
\[
\tilde U (\xi_1, \xi_2) \equiv - i \tilde K (\xi_1, \xi_2)\, \tilde W (\xi_1, \xi_2)
\]
 can be analytically continued to the domain 
$(H^- \setminus \{- k_1\}) \times (H^- \setminus \{- k_2 \})$. The residues of
$\tilde U$ at $\xi_1 = - k_1$ and $\xi_2 = - k_2$ are given by the following asymptotics:
\begin{equation}
\tilde U(\xi_1 , \xi_2) = 
\frac{1}{\xi_1 + k_1} 
\frac{\gamma(k_1 , k_2)}{\gamma(k_1, -\xi_2)\, (\xi_2 + k_2)} + O((\xi_1 + k_1)^0)
\label{eqI0115}
\end{equation}
\begin{equation}
\tilde U(\xi_1 , \xi_2) = 
\frac{1}{\xi_2 + k_2} 
\frac{\gamma(k_2 , k_1)}{\gamma(k_2, -\xi_1)\, (\xi_1 + k_1)} + O((\xi_2 + k_2)^0)
\label{eqI0116}
\end{equation}
The function is continuous at $P \times P$
\end{theorem}

\begin{remark}
Similarly to what has been done for Theorem~\ref{th:th1step1}, one can prove that 
$\tilde U$ is analytic on the parts of the boundary 
$(H^- \setminus \{ - k_1\}) \times P$ and $P \times (H^- \setminus \{ - k_2\})$.
However, we do not need this result and skip the corresponding argument. The continuity 
of $\tilde U$ on $P \times P$ is still important. It is understood in the sense that 
for any path in $(H^- \setminus \{ - k_1\})\times (H^- \setminus \{ - k_2\})$
ending at some point of $P \times P$ there exists a limit depending only on the 
ending point to which $\tilde U$ tends along this path.
\end{remark} 

\begin{proof} Use the central part of the equality (\ref{eqI0113}) to rewrite the formula (\ref{eq:stupidformula11}) as follows: 
\begin{eqnarray}
 \tilde U(\xi_1 , \xi_2) & = & 
\frac{ 1 }{4\pi^2 i \, \gamma(\xi_1, -\xi_2)}
\int \limits_P  
\mathd \xi_2'
\int \limits_{-\infty}^{\infty} 
\mathd \xi_1' 
\frac{
\gamma(\xi_1 , -\xi_2') \, \tilde K (\xi_1' , \xi_2')\, \tilde W(\xi_1' , \xi_2')}{
(\xi_1' - \xi_1) (\xi_2' - \xi_2)} \nonumber \\ 
&+& \frac{ \gamma(\xi_1, k_2) \, \gamma(k_2, k_1) }{
(\xi_1 + k_1)\, (\xi_2 + k_2) \, \gamma(k_2 , -\xi_1) \, \gamma(\xi_1, -\xi_2) },
\label{eqI0117}
\end{eqnarray}
and consider the integral 
\[
J(\xi_1 , \xi_2) \equiv
\int \limits_P  
\mathd \xi_2'
\int \limits_{-\infty}^{\infty} 
\mathd \xi_1' 
\frac{
\gamma(\xi_1 , -\xi_2') \, \tilde K (\xi_1' , \xi_2')\, \tilde W(\xi_1' , \xi_2')}{
(\xi_1' - \xi_1) (\xi_2' - \xi_2)}. 
\]
First, continue it from the domain $-\kappa < {\rm Im}(\xi_1) < 0$, 
$0 < {\rm Im}(\xi_2) < \kappa$ to the domain 
$-\kappa < {\rm Im}(\xi_1) < 0$, 
$-\kappa < {\rm Im}(\xi_2) < \kappa$. This continuation causes no problem since 
the only factor in the integrand depending on $\xi_2$
is $(\xi_2' - \xi_2)^{-1}$ and it is regular if $\xi_2' \in P$. 

Then consider $J(\xi_1, \xi_2)$
with $\xi_1$ and $\xi_2$ varying in the domain $H^-$. 
This will provide analytical continuation of $J$ into
$H^- \times H^-$. 
As in the proof of Theorem \ref{th:th1step1}, 
we need to make sure that the integrand is analytic.
Using Lemma \ref{lem:kappasign}, we conclude that the factor $\gamma(\xi_1,-\xi_2')$ is 
analytic for $\xi_1 \in H^-$, $\xi_2 \in P$. The polar factors are also 
regular for $\xi_{1,2} \in H^-$, $\xi_1' \in \mathbb{R}$, $\xi_2' \in P$, and since $\tilde{K}(\xi_1',\xi_2')\tilde{W}(\xi_1',\xi_2')$ does not depend on $\xi_{1,2}$, this proves the analyticity of $J$.

%on the 2-contour and so is $\tilde{W}(\xi_1',\xi_2')$ by Theorem \ref{th:th2step1}.Moreover, %n the 2-contour we cannot have $\xi_1'^2+\xi_2'^2=k^2$ and hence $\tilde{K}(\xi_1',\xi_2')$ %causes no problem either. Finally since no elements of $P$ are in $H^-$ and $(-\infty+i
%\kappa,\infty+i\kappa)$ lies strictly above $H^-$, we cannot have $\xi_2=\xi_2'$ or $\xi_1
%=\xi_1'$. Hence the integral is well behaved. 

The external factor $\gamma^{-1}(\xi_1,-\xi_2)$ is also analytic on $H^-\times H^-$ by 
Lemma~\ref{lem:kappasign} ($\sqrt{k_2 - \xi_1^2} \in H^+$, while $\xi_2 \in H^-$).
It should be noted that we changed the unknown from $\tilde W$ to $\tilde U$ in this theorem 
(comparatively, say, with Theorem~\ref{th:th1step1}) to prevent the external 
factor from having singularities in the domain of continuation. 

We now need to consider the explicit additive term. One can see once more due to Lemma \ref{lem:kappasign} that $\gamma(\xi_1,-\xi_2)$, $\gamma(k_2,-\xi_1)$ and $\gamma(\xi_1,k_2)$ are analytic. Hence the only singularities of this additive term in $H^- \times H^-$ are the simple poles $\xi_1 =- k_1$ and $\xi_2 = -k_2$, leading to the correct asymptotic behaviour (\ref{eqI0115}) and (\ref{eqI0116}).

The continuity on $P\times P$ can be proven as in Theorem~\ref{th:th1step1}. The problematic term comes from the polar factor $(\xi_2'-\xi_2)$ and can be dealt with by decomposing $\tilde{W}$ into two parts, a regular one and one that can be studied explicitly.

%{\bf RA, CHECK THIS LAST BIT}

\end{proof}

%%%%%%%%%%%%%%%%%%%%%%%%%%%%%%%%%%%%%%%%%%%%%%%%%%%%%%%%%%%%%%%%%%%%%%%%%

%%%%%%%%%%%%%%%%%%%%%%%%%%%%%%%%%%%%%%%%%%%%%%%%%%%%%%%%%%%%%%%%%%%%%%%%%%%%%%%%

\section{Additive crossing of branch lines} \label{sec:additivecrossingmain}

\subsection{Singular 2-lines}

Unlike the 1D case, in 2D complex analysis the basic singularities are not
isolated points, but (analytic) manifolds of real dimension~2 and of 
complex dimension~1. Below we call these manifolds {\it 2-lines}. 
Among these singularities, we are interested in 
simple poles and branch 2-lines. 

A primitive way to reveal the type of the singularity 
of a function $f$
is to introduce the 
local complex 
coordinates near some point of the singularity, one tangential and one transversal
$(\tau , \nu)$, 
fix the value of the tangential coordinate $\tau$ and see what
happens with $f$ as a function of a single variable $\nu$ at the singularity
$\nu = 0$.
For polar 2-lines it will be a pole, and for branch lines of order~$m$
it will be a branch point of order~$m$.\footnote{The rigorous definition of the order of a branch 2-line requires to think in terms of the fundamental group $\pi_1$ of manifolds (see \cite{Vol2}), which we omit here for brevity.}

According to Theorems~\ref{th:th1step1} and~\ref{th:th7}, we can 
state that $\xi_1 = - k_1$ and $\xi_2 = - k_2$ are polar 2-lines of 
$\tilde W(\xi_1 , \xi_2)$, 
and some fragments of the ``circle'' $\xi_1^2 + \xi_2^2 = k^2$ are branch 2-lines 
of order~2.  
These singularities are not unexpected. The polar lines are the poles of the 
right-hand side of the Wiener--Hopf equation, and the branch 2-line is that of the coefficient 
of the equation. The same behavior is demonstrated by the solution 
of the 1D Wiener--Hopf equation. 

What is new in the 2D case, is the appearance of the branch 2-lines
$\xi_1 = - k$ and $\xi_2 = -k$. Note that in Theorem \ref{th:th7} we could prove the analyticity of $\tilde U$ in a product of domains cut along the lines 
$h^-$
going 
from~$-k$. A considerably more sophisticated analysis 
shows that $\xi_1 = - k$ and $\xi_2 = -k$ are branch lines 
of order~2 (we will actually not use this fact in our consideration). Some important remarks about the link between the singularities of $\tilde W$
and the properties of the wave field can be found in Appendix~\ref{app:singvswavefield}. 
%, even we do not 
%use the fact that $\xi_1 = - k$ and $\xi_2 = -k$ are branch lines). 

We will now show that there is an important concept related to the lines,
$\xi_1 = - k$ and $\xi_2 = -k$, namely the concept of additive crossing 
of branch 2-lines. The next section is dedicated to this concept.

\subsection{The concept of additive crossing}\label{sec:additivecrossing}

Let $\eta_1$ and $\eta_2$ be some (local) complex variables, and let $\eta_1 = 0$ and 
$\eta_2 =0$ be branch 2-lines of a function $f(\eta_1, \eta_2)$.
Let $\eta_1 = 0$ be a branch 2-line of order $m_1$, and 
$\eta_2 =0$ be a branch 2-line of order~$m_2$. 

Consider the cuts $\chi_1$ and $\chi_2$ in the complex planes $\eta_1 $ and $\eta_2$, and define the left and the right shores of the cuts as 
shown in Fig.~\ref{fig06}.  

%%%%%%%%%%%%%%%%%%%%%%%%%%%%
 \begin{figure}[ht]
 \centerline{\epsfig{file=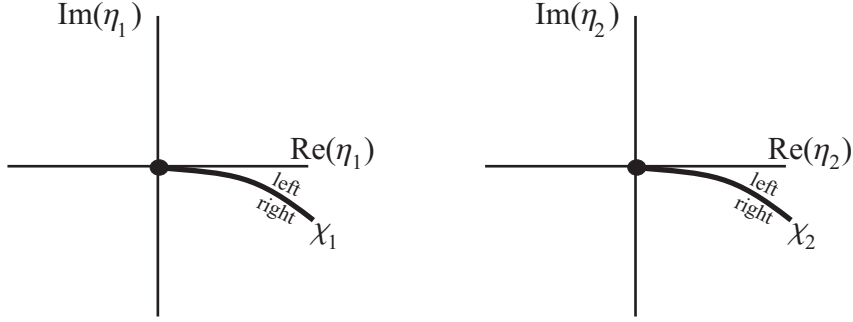}}
 \caption{Local coordinates $\eta_{1,2}$}
 \label{fig06}
 \end{figure}
%%%%%%%%%%%%%%%%%%%%%%%%%%%% 

Let $f$ be (locally) represented in the form
\begin{equation}
f(\eta_1 , \eta_2)  = \phi_1 (\eta_1 , \eta_2) + \phi_2 (\eta_1 , \eta_2)
\label{eqI0122}
\end{equation}
where $\phi_1$ has no branching about $\eta_2 =0$, and $\phi_2$ has no branching 
about $\eta_1 = 0$.
Let us now consider $\eta_1 \in \chi_1$ and $\eta_2 \in \chi_2$ and
let $\eta_{1,2}^{\rm l,r}$ be these variables taken on the left/right shore of the corresponding cut. It is possible to write down the values of $f$ on different 
shores of the cuts as follows: 
\[
f(\eta_{1}^{\rm l} , \eta_2^{\rm l})  =
\phi_1 (\eta_1^{\rm  l} , \eta_2) + \phi_2 (\eta_1 , \eta_2^{\rm l}) ,
\]
\[
f(\eta_{1}^{\rm l} , \eta_2^{\rm r})  =
\phi_1 (\eta_1^{\rm  l} , \eta_2) + \phi_2 (\eta_1 , \eta_2^{\rm r}) ,
\]
\[
f(\eta_{1}^{\rm r} , \eta_2^{\rm l})  =
\phi_1 (\eta_1^{\rm  r} , \eta_2) + \phi_2 (\eta_1 , \eta_2^{\rm l}) ,
\] 
\[
f(\eta_{1}^{\rm r} , \eta_2^{\rm r})  =
\phi_1 (\eta_1^{\rm  r} , \eta_2) + \phi_2 (\eta_1 , \eta_2^{\rm r}) .
\] 
When there is no upper index ``l'' or ``r'', it means that the function 
is regular on the cut, and the index does not matter. One can see directly that we have 
\begin{equation}
f(\eta_{1}^{\rm l} , \eta_2^{\rm l}) + 
f(\eta_{1}^{\rm r} , \eta_2^{\rm r})
=
f(\eta_{1}^{\rm l} , \eta_2^{\rm r}) + 
f(\eta_{1}^{\rm r} , \eta_2^{\rm l}).
\label{eq:additive}
\end{equation}
This property is the origin of the following definition. 

\begin{definition}
\label{def:additivecrossing} We say that a function $f(\eta_1,\eta_2)$ with branch 2-lines at $\eta_1=0$ and $\eta_2=0$ has the additive crossing property if the equation (\ref{eq:additive}) holds for some appropriate cuts.
\end{definition}

We have hence seen that if $f$ can be locally represented by (\ref{eqI0122}), then it has the additive crossing property. 

Conversely, assume now that a function $f$ has the additive crossing property, i.e. that (\ref{eq:additive}) is valid. If we also assume that $f$ can be represented as a Puiseux series in some vicinity of the origin as follows: 
\begin{equation}
f(\eta_1 , \eta_2) = 
\sum_{n_1, n_2 \in \mathbb{Z}} a_{n_1 , n_2} \eta_1 ^{n_1 / m_1} \eta_2 ^{n_2 / m_2},
\label{eq:powser}
\end{equation}
then, writing $\eta_{1,2}^{\rm l}=r_{1,2} e^{i \theta_{1,2}}$ and $\eta_{1,2}^{\rm r}=r_{1,2} e^{i (\theta_{1,2}+2\pi)}$, for each term of (\ref{eq:powser}), the additive crossing property (\ref{eq:additive}) leads to 
\begin{equation}
1 + e^{2 \pi i n_1 / m_1 } e^{2 \pi i n_2 / m_2 } = 
e^{2 \pi i n_1 / m_1 } +  e^{2 \pi i n_2 / m_2 }.
\label{eqI0123}
\end{equation}
This implies that 
\[
(1- e^{2 \pi i n_1 / m_1 })(1- e^{2 \pi i n_2 / m_2 }) = 0,
\]
and hence, each allowed term should have no branching either about $\eta_1=0$ or about $\eta_2 = 0$. Therefore $f$ can be locally represented by (\ref{eqI0122}), with each term of the series belonging either to $\phi_2$ or to $\phi_1$.

\begin{remark} Note that, strictly speaking, it is not necessary for $\eta_1=0$ and $\eta_2=0$ to be branch 2-lines of the function $f$ in order to define the concept of additive crossing. It is indeed sufficient to require $f(\eta_1 , \eta_2)$
to be analytic in a domain $(D \setminus \chi_1) \times (D\setminus \chi_2)$, 
where $D$ is some neighbourhood of the origin of a 1D complex plane, such that $f$ is continuous on the sides of the cuts (to make $f (\eta_1^{\rm l,r} , \eta_2^{\rm l,r})$ well defined).
Then the same relation (\ref{eq:additive}) will be called the additive crossing property 
of $f$. In our case we are going to establish the additive crossing property 
for
\begin{eqnarray*}
\eta_{1,2} = \xi_{1,2} + k, & f(\eta_1 , \eta_2)=\tilde U' (\eta_1 - k , \eta_2 - k), & \chi_{1,2} = h^- + k.
\end{eqnarray*}
\end{remark}

%%%%%%%%%%%%%%%%%%%%%%%%%%%%%%%%%%%%%%%%%%%%%%%%%%%%%%%%%%%%
\subsection{Deformation of the surface of integration for $u(x_1, x_2, 0)$} 

Consider the function $\tilde U'$ defined by (\ref{eqI0106}) and
introduce its inverse Fourier transform $v(x_1,x_2)$
%=\mathfrak{F}^{-1}[\tilde{U}'](x_1,x_2)$ 
as per (\ref{eqI0105}):  
\begin{equation}
v(x_1, x_2) \equiv \frac{1}{4\pi^2} 
\int \limits_{-\infty}^{\infty}
\int \limits_{-\infty}^{\infty}
\tilde U' (\xi_1 , \xi_2) e^{- i (\xi_1 x_1 + \xi_2 x_2)} 
\mathd \xi_1  \mathd \xi_2
\label{eqI0124}
\end{equation} 
and note that 
\[
v(x_1 , x_2) = u(x_1, x_2, 0) \quad \mbox{ for } \quad (x_1, x_2) \in (Q_2 \cup Q_3 \cup Q_4),
\]
 and 
\[ 
v(x_1 , x_2) = u(x_1, x_2, 0^+) - u^{\rm in} (x_1, x_2, 0)  
\quad \mbox{ for } \quad (x_1, x_2) \in Q_1 . 
\]
Therefore, according to the boundary condition (\ref{eq:BCDi}), it should be equal to zero if $x_1 > 0$ and $x_2 >0$. Let us consider $x_1 > 0$ and $x_2 > 0$ up to the end of this section.

The exponential 
factor of the Fourier transform, $e^{- i (\xi_1 x_1 + \xi_2 x_2) }$,
decays in the domain ${\rm Im}(\xi_1) < 0$, ${\rm Im}(\xi_2) <0$. The function $\tilde U'(\xi_1 , \xi_2)$ has two polar 2-lines $\xi_1 = -k_1$ and $\xi_2 = -k_2$. Let us introduce two small loops $\sigma_1$ and $\sigma_2$ encircling the points $-k_1$ and $-k_2$ in the anti-clockwise direction. 

In the integral  (\ref{eqI0124}), 
deform first the $\xi_1$ real axis into the contour $P$, and then the $\xi_2$ real axis into~$P$\footnote{In order to be precise, extra care should be taken when doing this transformation. Some additional steps involving intermediate contours should be added. However, for brevity, we do not provide all the details here.}. At each step, the usual 1D Cauchy's theorem is used to ensure that the integral preserves its value. The poles are taken care of by using the loops $\sigma_{1,2}$, resulting in: 
\begin{equation}
v(x_1, x_2) = \frac{1}{4\pi^2} 
\int_P
\int_P
\tilde U' (\xi_1 , \xi_2) e^{- i (\xi_1 x_1 + \xi_2 x_2)} 
\mathd \xi_1 \, \mathd \xi_2 +
S_1 + S_2 + S_{12},
\label{eqI0125}
\end{equation} 
where 
\begin{equation}
S_1 = -
 \frac{1}{4\pi^2} 
\int_P
\int_{\sigma_1}
\tilde U' (\xi_1 , \xi_2) e^{- i (\xi_1 x_1 + \xi_2 x_2)} 
\mathd \xi_1  \, \mathd \xi_2 \, ,
\label{eqI0126}
\end{equation}
\begin{equation}
S_2 = -
 \frac{1}{4\pi^2} 
\int_{\sigma_2}
\int_P
\tilde U' (\xi_1 , \xi_2) e^{- i (\xi_1 x_1 + \xi_2 x_2)} 
\mathd \xi_1  \, \mathd \xi_2 \, ,
\label{eqI0127}
\end{equation}
\begin{equation}
S_{12} = 
 \frac{1}{4\pi^2} 
\int_{\sigma_2}
\int_{\sigma_1}
\tilde U' (\xi_1 , \xi_2) e^{- i (\xi_1 x_1 + \xi_2 x_2)} 
\mathd \xi_1  \, \mathd \xi_2 \, .
\label{eqI0128}
\end{equation}
%{\bf RA, formally, here we should be more careful. It would be nice to introduce contour 
%$P'$ slightly wider than $P$ and make 4 steps instead of 2: (first both contours to $P'$, 
%then both contours from $P'$ to $P$). Also it would be nice to consider real axis not as
%a boundary of the domain, but as internal points. However, I do not want to overcomplicate 
%the paper before a reviewer asks.} OK, footnote added.
Consider the term $S_1$ and use (\ref{eqI0115}) and (\ref{eqI0106}) to compute the integral about the pole 
for each $\xi_2 \in P$ to get
\begin{equation}
S_1 = \frac{1}{2\pi i} \int_P
\left(
\frac{\gamma(k_1,k_2)}{\gamma(k_1 , -\xi_2)}
-1
\right) \frac{e^{i k_1 x_1 - i \xi_2 x_2}}{\xi_2 + k_2} \mathd \xi_2. 
\label{eqI0129}
\end{equation} 
Since the integrand has no branching at $\xi_2 = - k$ nor anywhere on $h^-$ (by Lemma \ref{lem:kappasign}), the integral is equal to zero and hence 
$S_1 = 0$. Similarly, we can show that $S_2 =0$. 

Finally, compute the double residue $S_{12}$.  One can see that the double residue 
of $\tilde U$ coming from (\ref{eqI0115}) is compensated with the double residue of the second term of (\ref{eqI0106}), and hence we have $S_{12} =0$. Thus, 

\begin{equation}
v(x_1, x_2) = \frac{1}{4\pi^2} 
\int_P
\int_P
\tilde U' (\xi_1 , \xi_2) e^{- i (\xi_1 x_1 + \xi_2 x_2)} 
\mathd \xi_1 \, \mathd \xi_2 .
\label{eqI0130}
\end{equation} 

The integral (\ref{eqI0130}) can be interpreted in terms of the additive crossing property of the 2-lines  $\xi_1 = -k$ and $\xi_2 = -k$. 
Consider $h^-$ as a (directed) contour going from $-k$ to~$- i \infty$.
The contour $P$ consists of two parts: one goes along the right shore of $h^-$ in the 
negative direction, and another goes along the left shore in the positive direction (see Fig.~\ref{fig02} (right) and Fig.~\ref{fig06} ). 
Thus, one can rewrite (\ref{eqI0130}) as follows: 
\begin{eqnarray}
v(x_1, x_2) &=& \frac{1}{4\pi^2} 
\int_{h^-}
\int_{h^-}
(
\tilde U' (\xi_1^{\rm l} , \xi_2^{\rm l}) 
+  
\tilde U' (\xi_1^{\rm r} , \xi_2^{\rm r}) 
-
\tilde U' (\xi_1^{\rm l} , \xi_2^{\rm r}) 
-  
\tilde U' (\xi_1^{\rm r} , \xi_2^{\rm l})
) 
\times \nonumber \\
& & 
e^{- i (\xi_1 x_1 + \xi_2 x_2)} 
\mathd \xi_1 \, \mathd \xi_2 .
\label{eqI0131}
\end{eqnarray} 
One can see that $v(\xi_1 , \xi_2)$ is equal to 0 if 
\begin{equation}
\tilde U' (\xi_1^{\rm l} , \xi_2^{\rm l}) 
+  
\tilde U' (\xi_1^{\rm r} , \xi_2^{\rm r}) 
-
\tilde U' (\xi_1^{\rm l} , \xi_2^{\rm r}) 
-  
\tilde U' (\xi_1^{\rm r} , \xi_2^{\rm l}) 
= 0,
\label{eqI0132}
\end{equation}
i.e.\ if the 2-lines
$\xi_1= - k$ and $\xi_2 = - k$ of the function $\tilde U'$ have the additive crossing property. Hence additive crossing implies 3/4-basedness. 

Conversely, a 2D uniqueness theorem (see Appendix \ref{app:laplace}) can be applied to the 
integral (\ref{eqI0131}) to get that if $v(x_1 , x_2) =0$ in $Q_1$ then 
(\ref{eqI0132}) is fulfilled. 
Thus, we have obtained an equivalence between the 3/4-basedness of $\tilde U'$ and the additive crossing property of the 2-lines
$\xi_1 = -k$ and $\xi_2 = -k$. 
This will allow us to reformulate the functional problem of Section~\ref{sec:functionalproblem}.

%%%%%%%%%%%%%%%%%%%%%%%%%%%%%%%%%%%%%%%%%%%%%%%%%%%%%%%%%%%%%%%%%%%%%%%%%%%%%%%%%%%%%%

\subsection{Reformulation of the functional problem}\label{sec:reformulationadditive}

The main result of the present paper is a reformulation of the functional problem from 
Section~\ref{sec:functionalproblem}. 
Consider the formulation of Theorem \ref{th:formulationSec2}.  
According to the content of this paper, the first two items of the theorem can be replaced by the following four conditions. The theorem remains valid after the replacement.

\begin{itemize}

\item[1'.]
$\tilde W(\xi_1 , \xi_2)$ is analytic in the domain 
\[
(\hat H^+ \times (\hat H^+ \cup H^- \setminus \{ -k_2 \})) 
\cup 
((\hat H^+ \cup H^- \setminus \{ -k_1 \}) \times \hat H^+).
\]

\item[1''.]
$\tilde W(\xi_1 , \xi_2)$ has poles at $\xi_1 = - k_1$, $\xi_2 = - k_2$, the residues of which are defined by (\ref{eqI0114}) and (\ref{eqI0114a}).

\item[2'.]
$\tilde U'(\xi_1 , \xi_2)$, as defined by (\ref{eqI0106}), is analytic in the domain 
\[
(H^- \setminus \{ -k_1 \}) \times (H^- \setminus \{ -k_2 \})
\]

\item[2''.] $\tilde U'(\xi_1,\xi_2)$ has the additive crossing property
for the 2-lines 
$\xi_1 = -k$ and $\xi_2 = - k$ with the cuts $h^-$.
\end{itemize}

Note that the residues (\ref{eqI0115}) and (\ref{eqI0116}) of $\tilde U'$ at $\xi_1 = -k_1$ and $\xi_2 = -k_2$ 
can be obtained by continuation of the residues of $\tilde W$. Hence, strictly speaking, it is not necessary to include them in the formulation of the functional problem.   

\section{Concluding remarks}
An important step of the usual 1D Wiener--Hopf method is to draw conclusions about the analyticity properties of unknown functions originally defined by half-range Fourier transforms. This is exactly what we have done here in a 2D context, and this is why we believe that this final reformulation is important. Indeed we drew some conclusions about the domain of analyticity of unknown functions originally defined by 1/4 and 3/4 range Fourier transforms. 

We established analyticity of the unknown functions $\tilde W$
and $\tilde U'$ in domains totally forming the $\mathbb{C} \times \mathbb{C}$
space (which agrees with the Wiener--Hopf concept), however, unfortunately, this space contains the branch 2-lines $\xi_1 = -k$, $\xi_2 = - k$. 
Almost nothing is known about this branching and hence, formally, at this level of understanding, the Liouville theorem (which is valid in $\mathbb{C}^2$, see e.g. \cite{Shabat2}) cannot be applied, and thus the Wiener--Hopf method cannot be completed. 
We plan to demonstrate in a subsequent paper that the additive crossing property of the branch 2-lines is in fact a very strong condition, and that some important physical features (such as the vertex asymptotics for example) can be recovered from it. 

\section*{Acknowledgements}

The work of A.V. Shanin has been supported by Russian Science Foundation grant
RNF~14-22-00042. R.C. Assier would like to acknowledge the support by UK EPSRC (EP/N013719/1).

%%%%%%%%%%%%%%%%%%%%%%%%%%%%%%%%%%%%%%%%%%%%%%%%%%%%%%%%%%%%

\bibliography{biblio}
\bibliographystyle{unsrt}%doipubmed}%harvard}% plain}%uabbrvnat}%natbib}%unsrtnat}%

\appendix
\counterwithin{figure}{section}

%%%%%%%%%%%%%%%%%%%%%%%%%%%%%%%%%%%%%%%%%%%%%%%%%%%%%%%%%%%%%%%%
\section{Proof of Theorem \ref{th:th3}} \label{app:proofth3}

In this appendix, for simplicity, we will use the notations of \cite{Assier2018}. That is, we will say that a function of the two variables $(\xi_1,
\xi_2)$ is a $\circ +$ (resp. $\circ -$) function if it is analytic in the UHP
(resp. LHP) of the $\xi_2$ plane, when considered a function of $\xi_2$ only.
Similarly, we say that such a function is a $+ \circ$ (resp. $- \circ$)
function if it is analytic in the UHP (resp. LHP) of the $\xi_1$ plane, when
considered a function of $\xi_1$ only. We can then naturally define a $+ -$
function as a function that is at the same time a $+ \circ$ and a $\circ -$
function. Similarly, it is possible to define $+ +$, $- +$ and $- -$
functions. These properties will be indicated as a subscript when necessary.

Rewrite the factorisation as follows:
\begin{eqnarray*}
 \widetilde{K} (\tmmathbf{\xi})=\widetilde{K}_{\circ +}(\tmmathbf{\xi}) \times \widetilde{K}_{\circ -} (\tmmathbf{\xi}),&
 \widetilde{K}_{\circ +}(\tmmathbf{\xi})=\frac{1}{\gamma (\xi_1, \xi_2)},&
 \widetilde{K}_{\circ -}(\tmmathbf{\xi})=\frac{1}{\gamma (\xi_1, - \xi_2)}.
\end{eqnarray*}
Starting from the functional equation, we get
\begin{equation}
  \tilde{K} (\tmmathbf{\xi}) \tilde{W} (\tmmathbf{\xi})=i \tilde{U}
  (\tmmathbf{\xi})= i \left( \frac{1}{(\xi_1 + k_1) (\xi_2 + k_2)} + \tilde{U}'
  (\tmmathbf{\xi}) \right)  \label{eq:thproofpart0}
\end{equation}
Now, since $\tilde{U}'$ is 3/4-based, there exists a function $\mathfrak{u}'$
such that\footnote{We have used the notation $\mathfrak{u}'$ for brevity, but in fact, we know from (\ref{eq:BCDi}), (\ref{eq:defUtilde}) and (\ref{eqI0104}) that $\mathfrak{u}'(x_1,x_2)=u(x_1,x_2,0)$.} $\tilde{U}' =\mathfrak{F}_{3 / 4} [\mathfrak{u}']$. We can hence introduce
the useful functions $\tilde{U}'_{+ -}$ and $\tilde{U}'_{-\circ}$ by
\begin{eqnarray*}
  \tilde{U}' (\tmmathbf{\xi}) & = & \mathfrak{F}_{3 / 4} [\mathfrak{u}']
  (\tmmathbf{\xi}) = \iint_{Q_2 \cup Q_3 \cup Q_4} \mathfrak{u}'
  (\tmmathbf{x}) e^{i\tmmathbf{x} \cdot \tmmathbf{\xi}} \, \mathd
  \tmmathbf{\xi}\\
  & = & \iint_{Q_4} \mathfrak{u}' (\tmmathbf{x}) e^{i\tmmathbf{x} \cdot
  \tmmathbf{\xi}} \, \mathd \tmmathbf{\xi}+ \iint_{Q_2 \cup Q_3}
  \mathfrak{u}' (\tmmathbf{x}) e^{i\tmmathbf{x} \cdot \tmmathbf{\xi}} \,
  \mathd \tmmathbf{\xi}= \tilde{U}'_{+ -} (\tmmathbf{\xi}) + \tilde{U}'_{-
  \circ} (\tmmathbf{\xi})
\end{eqnarray*}
Now, using the factorisation of $\tilde{K}$, we get
\begin{eqnarray}
  \frac{\tilde{K}_{\circ +} (\tmmathbf{\xi}) \tilde{W} (\tmmathbf{\xi})}{i} &
  =  & \frac{1}{\widetilde{K}_{\circ -} (\tmmathbf{\xi}) (\xi_1 + k_1)
  (\xi_2 + k_2)} + \frac{\tilde{U}'_{+ -}
  (\tmmathbf{\xi})}{\widetilde{K}_{\circ -} (\tmmathbf{\xi})} +
  \frac{\tilde{U}'_{- \circ} (\tmmathbf{\xi})}{\widetilde{K}_{\circ -}
  (\tmmathbf{\xi})}  \label{eq:proofstep1}
\end{eqnarray}
We will also use various sum-split Cauchy operators, namely $[.]_{\circ +}$,
$[.]_{\circ -}$, $[.]_{+ \circ}$ and $[.]_{- \circ}$ defined as follows for
any function $\tilde{\Phi} (\xi_1, \xi_2)$ analytic within a product of strips we
have
\begin{eqnarray*}
  {}[\tilde{\Phi}]_{\circ -} (\xi_1, \xi_2) = \frac{- 1}{2 i \pi} \int_{- \infty + i
  \kappa}^{\infty + i \kappa} \frac{\tilde{\Phi} (\xi_1, \xi_2')}{\xi_2' - \xi_2}
  \, \mathd \xi_2' & \tmop{and} & [\tilde{\Phi}]_{\circ +} (\xi_1, \xi_2) =
  \frac{1}{2 i \pi} \int_{- \infty - i \kappa}^{\infty - i \kappa} \frac{\tilde{\Phi}
  (\xi_1, \xi_2')}{\xi_2' - \xi_2} \, \mathd \xi_2' \\
  {}[\tilde{\Phi}]_{- \circ} (\xi_1, \xi_2) = \frac{- 1}{2 i \pi} \int_{- \infty + i
  \kappa}^{\infty + i \kappa} \frac{\tilde{\Phi} (\xi_1', \xi_2)}{\xi_1' - \xi_1}
  \, \mathd \xi_1' & \tmop{and} & [\tilde{\Phi}]_{+ \circ} (\xi_1, \xi_2) =
  \frac{1}{2 i \pi} \int_{- \infty - i \kappa}^{\infty - i \kappa} \frac{\tilde{\Phi}
  (\xi_1', \xi_2)}{\xi_1' - \xi_1} \, \mathd \xi_1'.
  \label{eq:Cauchysumsplitinalpha1}
\end{eqnarray*}
We can hence rewrite (\ref{eq:proofstep1}) as
\begin{eqnarray*}
  \frac{\tilde{K}_{\circ +} (\tmmathbf{\xi}) \tilde{W} (\tmmathbf{\xi})}{i} &
  = & \frac{1}{\widetilde{K}_{\circ -} (\tmmathbf{\xi}) (\xi_1 + k_1)
  (\xi_2 + k_2)} + \frac{\tilde{U}'_{+ -}
  (\tmmathbf{\xi})}{\widetilde{K}_{\circ -} (\tmmathbf{\xi})} + \left[
  \frac{\tilde{U}'_{- \circ} (\tmmathbf{\xi})}{\widetilde{K}_{\circ -}
  (\tmmathbf{\xi})} \right]_{\circ -} + \left[ \frac{\tilde{U}'_{- \circ}
  (\tmmathbf{\xi})}{\widetilde{K}_{\circ -} (\tmmathbf{\xi})}
  \right]_{\circ +},
\end{eqnarray*}
and, upon rearranging, we get
\begin{eqnarray}
  \frac{\tilde{K}_{\circ +} (\tmmathbf{\xi}) \tilde{W} (\tmmathbf{\xi})}{i} -
  \left[ \frac{\tilde{U}'_{- \circ} (\tmmathbf{\xi})}{\widetilde{K}_{\circ
  -} (\tmmathbf{\xi})} \right]_{\circ +} & = & \frac{\tilde{U}'_{+ -}
  (\tmmathbf{\xi})}{\widetilde{K}_{\circ -} (\tmmathbf{\xi})} + \left[
  \frac{\tilde{U}'_{- \circ} (\tmmathbf{\xi})}{\widetilde{K}_{\circ -}
  (\tmmathbf{\xi})} \right]_{\circ -} + \frac{1}{\widetilde{K}_{\circ -}
  (\tmmathbf{\xi}) (\xi_1 + k_1) (\xi_2 + k_2)}. \nonumber \\
& & \label{eq:proofth3part2}
\end{eqnarray}
The sum-split in the $\xi_2$ plane of the term involving the poles can be done
explicitly by the pole removal technique to get
\begin{eqnarray*}
  \frac{1}{\widetilde{K}_{\circ -} (\tmmathbf{\xi}) (\xi_1 + k_1) (\xi_2 +
  k_2)} & = & \frac{1}{(\xi_1 + k_1) (\xi_2 + k_2)} \left(
  \frac{1}{\widetilde{K}_{\circ -} (\tmmathbf{\xi})} -
  \frac{1}{\widetilde{K}_{\circ -} (\xi_1, - k_2)} \right)\\
  & + & \frac{1}{\widetilde{K}_{\circ -} (\xi_1, - k_2) (\xi_1 + k_1)
  (\xi_2 + k_2)},
\end{eqnarray*}
which finally leads (\ref{eq:proofth3part2}) to become
\begin{eqnarray}
  \frac{\tilde{K}_{\circ +} (\tmmathbf{\xi}) \tilde{W} (\tmmathbf{\xi})}{i} -
  \left[ \frac{\tilde{U}'_{- \circ} (\tmmathbf{\xi})}{\widetilde{K}_{\circ
  -} (\tmmathbf{\xi})} \right]_{\circ +} - \frac{1}{\widetilde{K}_{\circ -}
  (\xi_1, - k_2) (\xi_1 + k_1) (\xi_2 + k_2)} & = &  \nonumber\\
  \frac{\tilde{U}'_{+ -} (\tmmathbf{\xi})}{\widetilde{K}_{\circ -}
  (\tmmathbf{\xi})} + \left[ \frac{\tilde{U}'_{- \circ}
  (\tmmathbf{\xi})}{\widetilde{K}_{\circ -} (\tmmathbf{\xi})}
  \right]_{\circ -} + \frac{1}{(\xi_1 + k_1) (\xi_2 + k_2)} \left(
  \frac{1}{\widetilde{K}_{\circ -} (\tmmathbf{\xi})} -
  \frac{1}{\widetilde{K}_{\circ -} (\xi_1, - k_2)} \right). &  & 
  \label{eq:proofpart4}
\end{eqnarray}
One can see that the LHS of (\ref{eq:proofpart4}) is analytic in the UHP of
the $\xi_2$ plane, while its RHS is analytic in the LHP of the $\xi_2$ plane.
Application of Liouville's theorem in the $\xi_2$ plane implies that both sides are zero, leading
to
\begin{eqnarray}
  \tilde{W} (\tmmathbf{\xi}) & = & \frac{i}{\widetilde{K}_{\circ -} (\xi_1,
  - k_2) \tilde{K}_{\circ +} (\tmmathbf{\xi}) (\xi_1 + k_1) (\xi_2 + k_2)} +
  \frac{i}{\tilde{K}_{\circ +} (\tmmathbf{\xi})} \left[ \frac{\tilde{U}'_{-
  \circ} (\tmmathbf{\xi})}{\widetilde{K}_{\circ -} (\tmmathbf{\xi})}
  \right]_{\circ +} \nonumber\\
  & = & \frac{i \gamma (\xi_1, \xi_2) \gamma (\xi_1, k_2)}{(\xi_1 + k_1)
  (\xi_2 + k_2)} + \frac{\gamma (\xi_1, \xi_2)}{2 \pi} \int_{- \infty - i
  \kappa}^{\infty - i \kappa} \frac{\tilde{U}'_{- \circ} (\xi_1,
  \xi_2')}{(\xi_2' - \xi_2) \widetilde{K}_{\circ -} (\xi_1, \xi_2')}
  \, \mathd \xi_2' \nonumber\\
  & = & \frac{i \gamma (\xi_1, \xi_2) \gamma (\xi_1, k_2)}{(\xi_1 + k_1)
  (\xi_2 + k_2)} + \frac{\gamma (\xi_1, \xi_2)}{2 \pi} \int_{- \infty - i
  \kappa}^{\infty - i \kappa} \frac{\gamma (\xi_1, - \xi_2') \tilde{U}'_{-
  \circ} (\xi_1, \xi_2')}{(\xi_2' - \xi_2)} \, \mathd \xi_2'. \nonumber \\
 & & 
  \label{eq:thproofpart5}
\end{eqnarray}
Let us now note from (\ref{eq:thproofpart0}) that we have
\begin{eqnarray}
  \tilde{U}'_{- \circ} (\xi_1, \xi_2') & = & \left[ \frac{\tilde{K} (\xi_1,
  \xi_2') \tilde{W} (\xi_1, \xi_2')}{i} \right]_{- \circ} \nonumber\\
  & = & \frac{1}{2 \pi} \int_{- \infty + i \kappa}^{\infty + i \kappa}
  \frac{\tilde{K} (\xi_1', \xi_2') \tilde{W} (\xi_1', \xi_2')}{\xi_1' - \xi_1}
  \, \mathd \xi_1'. \label{eq:thproofpart6}
\end{eqnarray}
Combining (\ref{eq:thproofpart5}) and (\ref{eq:thproofpart6}), we obtain
\begin{eqnarray*}
  \tilde{W} (\tmmathbf{\xi}) & = & \frac{i \gamma (\xi_1, \xi_2) \gamma
  (\xi_1, k_2)}{(\xi_1 + k_1) (\xi_2 + k_2)}\\
  & + & \frac{\gamma (\xi_1, \xi_2)}{2 \pi} \int_{- \infty - i
  \kappa}^{\infty - i \kappa} \frac{\gamma (\xi_1, - \xi_2')}{(\xi_2' -
  \xi_2)} \left( \frac{1}{2 \pi} \int_{- \infty + i \kappa}^{\infty + i
  \kappa} \frac{\tilde{K} (\xi_1', \xi_2') \tilde{W} (\xi_1', \xi_2')}{\xi_1'
  - \xi_1} \, \mathd \xi_1' \right) \, \mathd \xi_2'\\
  & = & \frac{i \gamma (\xi_1, \xi_2) \gamma (\xi_1, k_2)}{(\xi_1 + k_1)
  (\xi_2 + k_2)}\\
  & + & \frac{\gamma (\xi_1, \xi_2)}{4 \pi^2} \int_{- \infty - i
  \kappa}^{\infty - i \kappa} \frac{\gamma (\xi_1, - \xi_2')}{(\xi_2' -
  \xi_2)} \left( \int_{- \infty + i \kappa}^{\infty + i \kappa}
  \frac{\tilde{K} (\xi_1', \xi_2') \tilde{W} (\xi_1', \xi_2')}{\xi_1' - \xi_1}
  \, \mathd \xi_1' \right) \, \mathd \xi_2',
\end{eqnarray*}
as required. The proof of the second formula can be obtained in a very similar
way by first splitting $\tilde{U}'$ as an integral over $Q_2$ and an
integral over $Q_3 \cup Q_4$, and then perform a sum split of the
functional equation in the $\xi_1$ plane instead.

%%%%%%%%%%%%%%%%%%%%%%%%%%%%%%%%%%%%%%%%%%%%%%%%%%%%%%%%%%%%%%%%%%%%%%%%%%%%%%%%%%%%%%%

\section{Singularities of the analytical continuation} \label{app:singvswavefield}

Consider the real plane $(\xi_1, \xi_2)$ and the function 
$\tilde W(\xi_1 , \xi_2)$ on this plane. This plane will be called 
the physical real plane of $(\xi_1, \xi_2)$. In the analytical continuation procedure we found the following 
singularities of $\tilde W$:

\begin{itemize}
\item
a part of the circle 
$\xi_1^2 + \xi_2^2 = k^2$, which is the branch 2-line of the coefficient
of the functional equation;
 
\item 
the branch 2-lines $\xi_1 = - k$ and  $\xi_2 = - k$ (the cuts $h^-$ start at $-k$); 

\item
the polar 2-lines $\xi_1 = - k_1$ and $\xi_2 = - k_2$.

\end{itemize}

Note that only the part of the circle shown in bold in Fig.~\ref{fig05} (left) is a singular set on the physical plane. 
All the rest of the circle belongs to the set of analyticity. This can be shown as follows. Let $\epsilon$ be small positive, and then take the limit $\epsilon \to 0$. 
Consider a vicinity of the point $\xi_1 = \cos (\vph)$, $\xi_2 = \sin (\vph)$, 
namely, consider the complex numbers 
\[
\xi_1 = \cos (\vph) + \mu_1, 
\qquad 
\xi_2 = \sin (\vph) + \mu_2, 
\]
where $\vph$ is real, and $\mu_{1,2}$ are small and chosen such that
\[
\xi_1^2 + \xi_2^2 =k^2= 1 + i \epsilon.
\]
In the linear approximation, this gives 
\begin{equation}
\cos(\vph) {\rm Im}(\mu_1) + \sin(\vph) {\rm Im}(\mu_2) = \epsilon / 2. 
\label{eqI0120}
\end{equation}

%%%%%%%%%%%%%%%%%%%%%%%%%%%%
 \begin{figure}[ht]
 \centerline{\epsfig{file=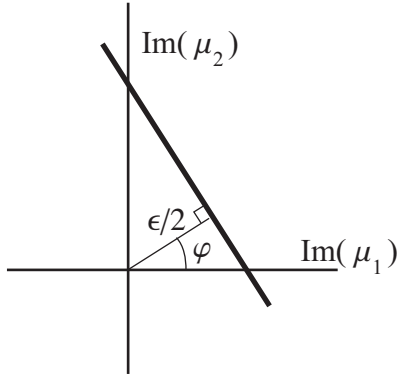}}
 \caption{Straight line structure resulting from (\ref{eqI0120})}
 \label{fig04}
 \end{figure}
%%%%%%%%%%%%%%%%%%%%%%%%%%%%

The set of values $\text{Im}(\mu_{1,2})$ obeying (\ref{eqI0120}) is a line (see Fig.~\ref{fig04}), the slope of which is determined by~$\vph$. One can see that if 
$-\pi/ 2 < \vph < \pi$ then this line passes through the zone with 
${\rm Im}(\xi_1)={\rm Im}(\mu_1) > 0$, ${\rm Im}(\xi_2)={\rm Im}(\mu_2) > 0$. Thus, some singular points appear in the 
zone of prescribed analyticity of $\tilde W$, which is prohibited. Hence, only the part of the circle with $-\pi < \vph < -\pi / 2 $ may belong to the 
common part of the singular set of $\tilde W$ and the physical real plane. 
This part of the circle is shown in Fig.~\ref{fig05} (left).   

%%%%%%%%%%%%%%%%%%%%%%%%%%%%
 \begin{figure}[ht]
 \centerline{\epsfig{file=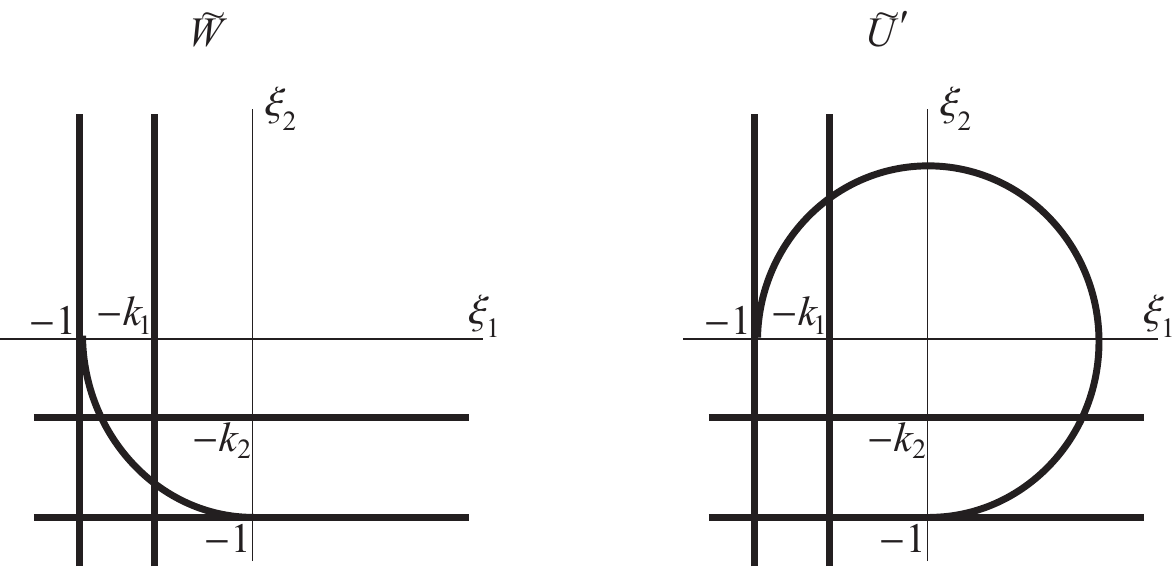}}
 \caption{Singularities of $\tilde W$ (left) and $\tilde U'$ (right)}
 \label{fig05}
 \end{figure}
%%%%%%%%%%%%%%%%%%%%%%%%%%%%

The branching of $\tilde W$ on the part of the circle with $-\pi < \vph < -\pi / 2 $
can be used to find the asymptotics of the function  
\[
\frac{\ptl u}{\ptl n}
 (- R \cos (\vph) , - R \sin (\vph) , 0^+) 
 \quad 
 \mbox{ as } \quad R \to \infty.
\] 
For this, the inverse Fourier transform and the saddle point method can be used in a standard way. Note that 
this function is equal to zero for $-\pi/2 < \vph < \pi$, so the absence of branching at this 
part of the circle is quite natural. 

As we have mentioned, the circle $\xi_1^2 + \xi_2^2 = k^2$ corresponds to the singular set of the 
coefficient of the functional equation, so it can be expected that this set will be 
a singular set of the solution. This situation reminds the classical 1D Wiener--Hopf method. 
However, the presence of the branch 2-lines $\xi_1 = - k$, $\xi_2 = - k$ is something new that 
appears only in the 2D case. Physically, however, these branch lines are understandable. 
They correspond to the edge singularities of the field $u$ having the 
phase dependencies $\sim e^{i k x_2}$ for the edge $x_1 =0$, and 
$\sim e^{i k x_1}$ for  $x_2 =0$. Thus, in some sense,  this corresponds to the edge singularities of the field scattered by the vertex. 

Finally, the polar 2-lines $\xi_1 = - k_1$, $\xi_2 = - k_2$ correspond to the cones of rays scattered 
by the edges. The crossing point $(-k_1 , - k_2)$ is very important and corresponds 
to the incident and the reflected plane waves. 

Let us now consider the singularities of the function $\tilde U'$ defined by (\ref{eqI0106}). 
This function is 3/4-based, so its inverse Fourier transform (\ref{eqI0105}) 
should be equal to zero
on the quadrant $Q_1$. 
Note that when such a transform is made in $(x_1 , x_2) \in Q_1$, the ``contour of integration''
should be deformed into the domain ${\rm Im}(\xi_1) < 0$, ${\rm Im}(\xi_2) < 0$, 
i.e.\ the continuation established by Theorem \ref{th:th7} plays a crucial role.

Three important notes should be made about the singularities of $\tilde U'$.

\begin{enumerate}

\item
The polar sets $\xi_1 = -k_1$ and $\xi_2 = -k_2$ are intersecting within the domain $H^-\times H^-$. This can potentially produce an exponential term, but it is easy to check that the double 
residue of  (\ref{eqI0115}), (\ref{eqI0116}) is compensated by the explicit term 
of (\ref{eqI0106}) since $\gamma(k_2,k_1)/\gamma(k_2,-k_1)=1$. 

\item
Function $\tilde U'$ is regular at the points of the circle 
$\xi_1^2 + \xi_2^2 = k^2$ corresponding to $-\pi < \vph < -\pi /2$. The physical reasoning is as follows. Any branching at this point of the circle would lead to a field $u$ with asymptotics 
$\sim \exp \{ i k \sqrt{x_1^2 + x_2^2} \}$ on $Q_1$, which is prohibited by the boundary condition 
(\ref{eq:BCDi}). The mathematical reasoning is based on Theorem \ref{th:th7}. If $\epsilon \ne 0$ then the
points belonging to this part of the circle belong to the domain 
$(H^- \setminus \{-k_1 \}) \times (H^- \setminus \{-k_2 \})$. Namely, 
they are points 
\[
( - k \cos(\vph) , -k \sin (\vph))
\]  
for some (narrow) complex strip surrounding the real segment $-\pi < \vph < -\pi /2$. %{\bf AS: again, why the minus sign? also how do you prove that about the small strip, is it just a one sided strip in the LHP of $\varphi$ plane? Not yet convinced by the argument on this bullet point, let's discuss.}
According to Theorem \ref{th:th7}, at these points $\tilde U'$ should be regular. Thus, 
$\tilde U'$ is allowed to be singular only at the points shown in Fig.~\ref{fig05}~(right).

\item
The crossing of branch lines $\xi_1 = -k$, $\xi_2 = -k$  belongs 
to the domain ${\rm Im}(\xi_1) < 0$, ${\rm Im}(\xi_2) < 0$. Potentially, such a crossing 
can lead to a wave component with the asymptotics 
$\sim \exp \{ i k(x_1 + x_2) \}$. Obviously, such a wave component does not exist. 
Thus, the crossing of branch lines should cause no field terms. As shown in Section \ref{sec:additivecrossing}, this characteristic of $\tilde U'$ is linked to the concept of additive crossing.
%For example, this 
%may happen when locally near this point function $\tilde U'$ can be represented as 
%a sum of two analytical functions
%\begin{equation}
%\tilde U' (\xi_1 , \xi_2) = \phi_1(\xi_1, \xi_2) + \phi_2(\xi_1, \xi_2), 
%\label{eqI0121}
%\end{equation} 
%where $\phi_1$ has branching only at $\xi_1 = - k$, and $\phi_2$ has branching only at 
%$\xi_2 = -k$, i.e.\ the crossing of these branch lines is additive. Below we give a %definition of 
%additive crossing and prove that $\tilde U'$ has this property. 
\end{enumerate}

%%%%%%%%%%%%%%%%%%%%%%%%%%%%%%%%%%%%%%%%%%%%%%%%%%%%%%%%%%%%%%%

\section{Laplace-like uniqueness theorem}\label{app:laplace}

\begin{theorem}[1D uniqueness theorem]\label{th:B1}
Let $L$ be a smooth curve without self-crossings in the $z$ complex plane.
Let $L$ start at the origin and end at infinity while lying within the sector $\beta_1 <{\rm Arg} (z) < \beta_2$ such that
$\beta_2 - \beta_1 < \pi$. 
Let $L$ be ``simple'' in the sense that there exists a constant $A>0$ such that for any $r_1$ and $r_2$ with $r_2>r_1>0$, the length of the curve $L$ within the annulus $r_1\leq r \leq r_2$ (in polar coordinates) is bounded by $A(r_2-r_1)$.
    Let $f(z)$ be a smooth function defined on $L$ and decaying as $|z| \to \infty$. Finally, let us assume that for all $s$ such that $ -\beta_1 <{\rm Arg}(s) < \pi - \beta_2$ we have
\begin{equation}
F(s) \equiv \int_L f(z) e^{i s z} \mathd z =0.
\label{eq:A101}
\end{equation}
Then $f(z) \equiv 0$.
\end{theorem}

{\bf Note.}
If $\beta_1 = \beta_2 = 0$ then (\ref{eq:A101}) is the Laplace transform. One can take the inverse (Mellin) transform and, thus, the theorem is trivial.

%%%%%%%%%%%%%%%%%%%%%%%%%%%%
 \begin{figure}[ht]
 \centerline{\epsfig{file=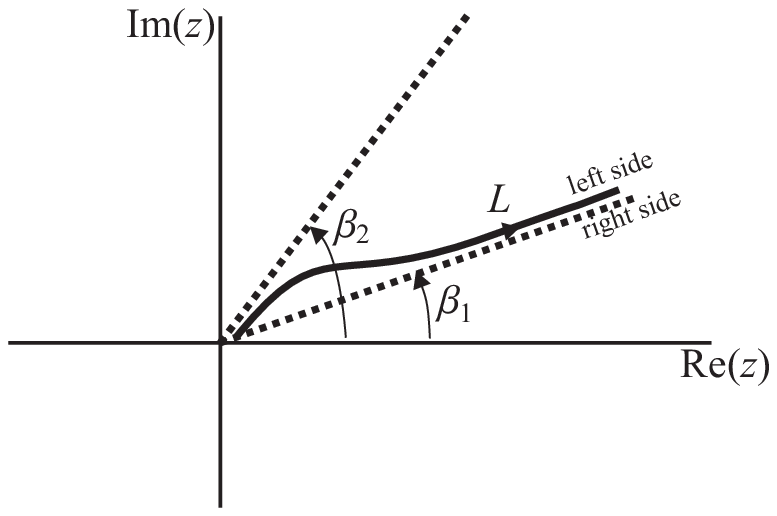}}
 \caption{Contour $L$}
 \label{figA01}
 \end{figure}
%%%%%%%%%%%%%%%%%%%%%%%%%%%%

\begin{proof} The proof will consist of three main steps.

{\bf Step 1.} 
Consider $L$ as a cut of the $z$ complex plane and define the left and right shores of the cut as in Fig.~\ref{figA01}. Construct a function $y(z)$ analytic in $\mathbb{C} \setminus L$ 
such that 
\begin{equation}
f(z) = y^{\rm r} (z) - y^{\rm l}(z) , \qquad z \in L,
\label{eq:A102}
\end{equation} 
where $y^{\rm l}$ and $y^{\rm r}$ are the limiting values of $y$ on the left and 
right shores of $L$ respectively. Such a function is given by the 
Sokhotsky formula
\begin{equation}
y(z) = \frac{i}{2\pi} \int_L \frac{f(z') \mathd z'}{z'-z} .
\label{eq:A103}
\end{equation}
 
{\bf Step 2.} 
Let $y_1 (z)$ and $y_2 (z)$ be the restrictions of $y(z)$ onto the lines ${\rm Arg}(z) = \beta_1$
and ${\rm Arg}(z) = \beta_2$ respectively. 
Let $Y_1 (s)$ and $Y_2 (s)$ be defined by 
\begin{equation}
Y_{1,2} (s) \equiv  e^{i \beta_{1,2}} \int \limits_0^{\infty} y_{1,2}(\tau e^{i \beta_{1,2}}) 
\exp \{ i s \tau e^{i \beta_{1,2}} \} \mathd \tau .
\label{eq:A104}
\end{equation}

The functions $Y_1$ and $Y_2$ are analytic in the domains where the integrals converge due to the exponential factor, i.e.\ in the sectors $-\beta_1 < {\rm Arg} (s) < \pi -\beta_1$, 
$-\beta_2 < {\rm Arg} (s) < \pi -\beta_2$, respectively. These sectors have the common part 
\[
{\mathcal S}: \quad -\beta_1 < {\rm Arg}(s) < \pi - \beta_2
\] 

Due to Cauchy's theorem and the analyticity property of $y$, for $s \in \mathcal S$ we have 
\begin{equation}
\int_L y^{\rm r} (z) e^{i s z} dz = Y_1(s), 
\label{eq:A105}
\end{equation}
\begin{equation}
\int_L y^{\rm l} (z) e^{i s z} dz = Y_2(s). 
\label{eq:A106}
\end{equation}

But by hypothesis,
\[
Y_1 (s)-Y_2(s) = F(s) = 0\qquad \mbox{ for }s \in \mathcal{S}.
\]
This means that we can analytically continue $Y_1(s)$ and $Y_2(s)$ into a function $Y (s)$ that is analytic on the sector $-\beta_2 < {\rm Arg}(s) < \pi - \beta_1$. 

%%%%%%%%%%%%%%%%%%%%%%%%%%%%
 \begin{figure}[ht]
 \centerline{\epsfig{file=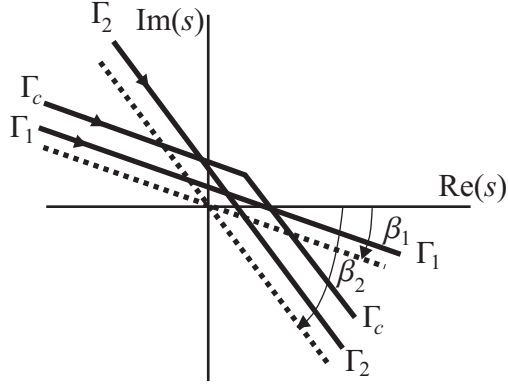}}
 \caption{Contours $\Gamma_1$, $\Gamma_2$, and $\Gamma_c$}
 \label{figA02}
 \end{figure}
%%%%%%%%%%%%%%%%%%%%%%%%%%%%

{\bf Step 3.} Let us now introduce the contours $\Gamma_1$ and $\Gamma_2$ as shown in Fig.~\ref{figA02}. 
One can reconstruct $y_1 (z)$ and $y_2 (z)$ by Mellin transform: 
\begin{equation}
y_{1,2} (z) = \frac{1}{2\pi} \int_{\Gamma_{1,2}} Y (s) e^{- i s z} \mathd s .
\label{eq:A107}
\end{equation}

Introduce the contour $\Gamma_c$ as shown in Fig.~\ref{figA02}. 
Note that due to Cauchy's theorem in (\ref{eq:A107}) one can deform $\Gamma_{1}$ or 
$\Gamma_2$ into $\Gamma_c$ to get: 
\begin{equation}
y_{1,2} (z) = \frac{1}{2\pi} \int_{\Gamma_c} Y (s) e^{- i s z} \mathd s .
\label{eq:A108}
\end{equation}
Moreover, remembering that $\beta_2-\beta_1<\pi$, it is easy to check that the integral of (\ref{eq:A108})
defines a function analytic in the sector $\beta_1 <{\rm Arg}(z) < \beta_2$. This function is 
hence a common analytical continuation of $y_1$ and $y_2$, implying that $y^{\rm r}(z) = y^{\rm l}(z)$, and finally $f(z) = 0$, as required.
\end{proof}

It is now possible to formulate and prove the 2D analog of this theorem as follows. 

\begin{theorem}[2D uniqueness theorem]
Let $L$ obey conditions of Theorem \ref{th:B1}. Let $f(z_1, z_2)$ be a smooth function defined on $L \times L$ and decaying 
as $|z_1|+ |z_2| \to \infty$. Let us assume that for all $s_1 , s_2$ such that 
$ -\beta_1 <{\rm Arg}(s_{1,2}) < \pi - \beta_2$, we have
\begin{equation}
F(s_1, s_2) \equiv \int _{L} \int_L  f(z_1 , z_2) e^{i (s_1 z_1 + s_2 z_2)} 
\mathd z_1 \, \mathd z_2  =0.
\label{eq:A109}
\end{equation}
Then $f(z_1, z_2) \equiv 0$.
\end{theorem}

\begin{proof}
Let us start by defining the function $g(z_2)$ by
\begin{eqnarray}
g(z_2) & = & \int_L f(z_1,z_2)e^{is_1z_1}\,\mathd z_1. \label{eq:lastone}
\end{eqnarray}
The hypothesis (\ref{eq:A109}) of the theorem can be rewritten as
\begin{eqnarray}
\int_L g(z_2)e^{is_2z_2}\, \mathd z_2 &=& 0.
\end{eqnarray}
The function $g$ can be shown to be smooth and decaying along $L$ and hence, using Theorem \ref{th:B1}, this implies that $g(z_2)=0$ for all $z_2\in L$. Now the integral in (\ref{eq:lastone}) is equal to zero. We can hence apply Theorem \ref{th:B1} one last time to prove that $f(z_1,z_2)=0$, as required.
\end{proof}

%%%%%%%%%%%%%%%%%%%%%%%%%%%%%%%%%%%%%%%%%%%%%%%%%%%%%%%%%%%%%%%

%\section{Proof of simple regularity conjecture}

\end{document}